 \documentclass[draft]{article}

\usepackage{amsmath,amsfonts,amsthm,amssymb,amscd,cancel,color,mathabx}
\usepackage{enumitem}
\usepackage{verbatim}
\usepackage[dvipdfmx]{graphicx}
\usepackage{ulem,color}

\usepackage{cleveref}

\renewcommand{\Re}{\mathop{\rm Re}\nolimits}

\def\S{\mathhexbox278}

\theoremstyle{plain}
\newtheorem{theorem}{Theorem}[section]
\newtheorem{lemma}[theorem]{Lemma}
\newtheorem{proposition}[theorem]{Proposition}

\theoremstyle{definition}

\theoremstyle{remark}

\newtheorem{notation}[theorem]{Notation}

\newcommand{\R}{{\mathbb R}}

\newcommand{\N}{{\mathbb N}}

\def\im{{\rm i}}

\newcommand{\C}{\mathbb{C}}

\def\({\left(}
\def\){\right)}
\def\<{\left\langle}
\def\>{\right\rangle}
\newcommand{\sech}{{\mathrm{sech}}}

\newcommand{\supp}{{\mathrm{supp}\ }}
    
\newcommand{\Span}{{\mathrm{Span}}}

\makeatletter

\numberwithin{equation}{section}

\begin{document}

\title{The virial--Kato method and the asymptotic stability of solitons of  the  generalized Benjamin, Bona and Mahony   equation }

\author{Scipio Cuccagna, Daniele Del Santo  and Masaya Maeda   }
\maketitle


\begin{abstract}
We prove that every orbitally stable soliton of the pure-power generalized Benjamin--Bona--Mahony (gBBM) equation is locally asymptotically stable. For both nonlinearities $f(u)=|u|^{p-1}u$ and $f(u)=|u|^p$, the result holds for every exponent $p>1$ throughout the orbital stability range $c>c^*(p)$. Solutions starting sufficiently close to one of these solitons in the energy space converge locally in a moving frame to a nearby soliton and the limiting soliton has a speed close to that of the initial soliton. This establishes, for pure-power gBBM, the local convergence predicted by Bona and Soyeur's asymptotic stability conjecture and addresses an open problem raised by Martel.

The proof uses Kato smoothing to close the high-energy virial estimates. The required smoothing estimates rely on the absence of nonzero embedded eigenvalues of the linearized operator. We establish this spectral property for every $p>1$ and every soliton speed $c>1$, including speeds for which the soliton is orbitally unstable. This settles a Pego--Weinstein spectral conjecture in the pure-power gBBM case.
\end{abstract}

\section{Introduction}

We consider on the line $\R $  the  generalized  Benjamin, Bona and Mahony  (gBBM)     equation,  also called generalized Regularized Long-Wave equation, 
\begin{align}\label{eq:BBM1}&
  (1   -\partial _x ^2)    \partial _t     u          +  \partial_x   \left (    u +  f(u) \right ) =0 \text{ where $f(u)=|u|^{p-1}u$   or $f(u)=|u|^{p } $,}
\end{align}
which,  for    $R_A(z) =(A-z)^{-1}$   the resolvent of an operator $A$,    we can write equivalently 
\begin{align}\label{eq:BBM2}&
    \partial _t     u          =-  \partial_x  R_{-\partial _x ^2} (-1)\left (    u +  f(u) \right )  .
\end{align}
  This equation was introduced  by  Peregrine  \cite{Peregrine1966,Peregrine1967}
and by   Benjamin et al. \cite{BBM1972}, as an alternative  to the Korteweg--de Vries equation.  
There are  invariants 
\begin{align}
   \label{eq:energy}E( {u}) &=      \int _{\R } \left (   \frac{u^2}2 +F(u) \right )      dx  \, , \text{ for $F'=f$ and $F(0)=0$} \\ \label{eq:mass}M( {u}) &=  2^{-1} \<  (1- \partial _x ^2)  u , u \>       ,
\end{align}
as is formally shown,  \eqref{eq:mass} by multiplying \eqref{eq:BBM1} by $u$ and then integrating in $x$ on the line,  resp.  
\eqref{eq:energy}   by multiplying \eqref{eq:BBM2} by $ u +  f(u)$.   The Cauchy problem is globally well posed in  
$H^1(\R )$, see  \cite{BBM1972} and, for a more general class of models,   \cite{SouganidisStrauss90}. In fact \eqref{eq:BBM2} is an ordinary differential equation, where the vector field in the right hand side is locally Lipschitz in $H^1(\R )$, and global existence is guaranteed by   $M(u(t))=M(u_0)$  where $u(0)=u_0$.

For $c>1$ there are solitons  ${Q}_c(x-ct-x_0)$ of the form
\begin{align}\label{eq:QQ1}
 Q_c( x):=(c-1)^{\!1/(p-1)}Q\left (  \sqrt{\frac{c-1}{c}}\,x \right ) \text{ with }
  Q( x) :=  \left(\frac{(p+1)}{2}\right)^{\!1/(p-1)}
 \sech^{\,2/(p-1)}\!\left(\frac{p-1}{2}
 \,x\right) 
\end{align}
which satisfy 
\begin{equation}\label{eq:soliton_eq}
	-cQ_c''+(c-1)Q_c- f(Q_c)=0   \text{  and }- Q ''+ Q - f(Q )=0 .
	\end{equation}
From Weinstein  \cite[Theorem 6.3]{weinCPDE1987}, see also Souganidis and Strauss    \cite{SouganidisStrauss90}, 
 we have the following.
 \begin{theorem}[Orbital Stability]\label{thm:orbstab} Let $p>1 $.  Then, for  
 \begin{align}\label{eq:c0p}
    c^*(p)=
\begin{cases}
1, & 1<p\leq 5,\\[2mm]
\displaystyle
1+\frac{p-5}{4}
\frac{\sqrt{2(p+3)}}{p-1+\sqrt{2(p+3)}},
& p>5 ,
\end{cases}
    \end{align}
  for any   $ c> c^*(p)$,  any $x_0\in \R$  and any $\epsilon >0$ there exists  a   $\delta_0(c, \epsilon  ) >0$ such that for all $0<\delta<\delta_0(c,  \epsilon)$ for any  $ {u} _{0}\in D _{ {H}^1(\R )} \left (  {Q}_c (\cdot -x_0), \delta    \right )    $   for the corresponding solution $u\in C ^0 ([0, +\infty ),   {H}^1 (\R ))$    there is a function $x(t)$ such that
\begin{equation}\label{eq:orbstab}
\sup_{t\geq 0} \| {u} -  {Q}_c (\cdot -x(t)) \|_{ {H}^1(\R )} \leq  \epsilon .
\end{equation}

    \end{theorem}

Very important 
to us is  the linearized operator  
\begin{align}\label{eq:Lin}
   & \mathcal{L}_{[c]}: =(1-\partial _x ^2) ^{-1} \partial_xL_c,
  \text{ where }
 L_c:=-c\partial_x^2+c-1-f'(Q_c )  .
\end{align}
Pego and Weinstein \cite[p.~53]{PegoWei1} conjectured the absence of nonzero embedded eigenvalues for several dispersive equations, including gBBM. 
The following theorem settles their conjecture in the pure-power gBBM case.

\begin{theorem}[Absence of nonzero  embedded eigenvalues]\label{prop:noemb_all}
     For any $p>1$ and $c>1$, $\sigma_{\mathrm{p}}(\mathcal L_{[c]})\cap  \im \R\setminus \{0\} =\emptyset$.
\end{theorem}
 Theorem \ref  {prop:noemb_all}, proved in Sect.  \ref{sec:noemb}, generalizes the result proved by   Martel \cite[Theorem 2]{MartelLinear} for the case $p=2$. Notice that  Theorem \ref  {prop:noemb_all} holds for all solitons with $p>1$ and $c>1$, both for orbitally  stable and unstable solitons.  It is a very general result, dealing with    a question beyond the usual index theorems, for which we refer to Lin and Zeng \cite{LZ22Memoiers}, see later in the proof of Proposition  \ref{prop:noembeig1}.
 Our proof uses Frobenius expansions for a third-order
Fuchsian equation with three regular singular points,
together with analytic continuation and Liouville's theorem,   following arguments used for the  quintic NLS in Perelman \cite[Sect.2.1.3]{Perelman01AHP}, later generalized by Krieger and Schlag \cite{KS06JAMS}. 

Bona and Soyeur \cite[\S~6, p.~468]{BonaSoyeur1994} conjectured that, for a class of dispersive equations including gBBM, a small perturbation of a stable solitary wave evolves into a nearby solitary wave and a small remainder that separates from it and is left behind. 
In his survey, Martel \cite[\S~3.5, p.~98]{MartelSeMA2005} noted that the asymptotic stability of solitary waves for the generalized BBM equation remained an open problem.
The following theorem addresses this question by establishing the local asymptotic stability predicted by Bona and Soyeur for pure-power gBBM throughout the orbital stability range $c>c^*(p)$. 
Its proof uses Theorem~\ref{prop:noemb_all} to obtain the Kato smoothing estimates needed to close the virial estimates.

\begin{theorem}[Asymptotic Stability]
  \label{thm:main} Let $ a>0$  and $p>1$. Then for any $c_0> c^*(p)$ and any small $\varepsilon >0$   there exists a  $\delta_0  >0$   such that for any 
  solution $u\in C^0([0,+\infty) ,H^1(\R))$ of  \eqref{eq:BBM2}   and  $ {u} _{0}\in D _{ {H}^1} \left (  {Q}_{c_0} (\cdot -x_0), \delta    \right )    $  for $0<\delta<\delta_0(c_0 ,\varepsilon)$  
  there exist  functions
  $ c \in C^1 ([0,+\infty), (c_0-\varepsilon,c_0+\varepsilon) ),$ $D \in C^1 ([0,+\infty) , \R ) $ and
  $v\in C^0 ([0,+\infty) ,H^1(\R))$
     such that
  \begin{align} \label{eq:asstab1}
    & u(t)=    Q _{c (t)} (\cdot -D(t) ) + v(t, \cdot -D(t) ) \text{   with}
 \\&  \label{eq:asstab2}   \int _0^{+\infty} \|  e^{- a|x|}   v (t ) \| _{H^1(\R )}^2 dt <  \varepsilon   \\\label{eq:asstab20}
   &\lim _{t\to +\infty  }  \|  e^{- a| x|}   v (t ) \| _{L^2(\R )}     =0 
\end{align}
and there exists $c_+>0$ s.t.
\begin{align}
&\lim _{t\to +\infty}c (t)= c _+ . \label{eq:asstab3}
\end{align}
\end{theorem}

The proof of the various nonlinear estimates is simpler if $p\ge 2$. In the proof of  the nonlinear estimates we will consider the case  $1<p\le 2$, which is more delicate because $x\to f(x)$ is less regular. All the nonlinear estimates have simpler versions for $p>2$.  
An analogue of the Theorem \ref{thm:main} for  the quadratic nonlinearity $f(u)=u^2$ and for a smaller set of solutions,  was proved by Miller and Weinstein \cite{MillerWeinsteinBBM}, based on the framework in Pego and Weinstein \cite{PegoWei2},   
 for all velocities $c>1$ except for those in a singular set, which at most is countable and, if infinite, is a divergent sequence $c_n\xrightarrow{n\to +\infty}+\infty $.  Theorem \ref{thm:main}  for  $f(u)=u^2$ follows by  El Dika \cite{ElDIkaasstab} and Mizumachi \cite{MizumachiBBM},  
who independently  combined the theory developed for gKdV in  Martel and Merle \cite{MaMeARMA2001,MaMeNonlinearity2005,MaMegeneral2008,MaMeRefined2008} with the theory of the linearized operator in Pego and Weinstein \cite{PegoWei2}.   The exceptional set is empty  because of Theorem \ref{prop:noemb_all} which for $p=2$ was proved by Martel \cite{MartelLinear} as we mentioned above.

 Theorem  \ref{prop:noemb_all} is the key  for applying our virial--Kato general methodology which we have already applied for  1 D NLS in    \cite{CM24D1, CM243, CM24D4}, generalized KdV \cite{CM25D1}, Euler--Poisson \cite{BaCuMa} and 
   generalized  good Boussinesq equations \cite{BaCuDeSaMa}. The virial  inequalities encounter  a well known  obstruction   due to some lower order terms. This is very clear in the work of  Martel and Merle, where  in  \cite{MaMeARMA2001} for gKdV there is a detailed case by case analysis of the positivity of a certain quadratic form, only for $p=2,3,4$. Later in \cite{MaMegeneral2008}, they  introduce a  duality argument, specific for gKdV--type equations,   based on the duality argument in Martel \cite{MartelLinear}. In our series of papers, we have proposed a different solution to this difficulty, different also from that of  El Dika \cite{ElDIkaasstab} and Mizumachi \cite{MizumachiBBM},
    based on the observation that the obstructive terms   are exactly of the type considered  by  Kato \cite{katoMA1966}, so that the use of Kato smoothing to close the virial inequalities  appears   natural.

   Our method is largely model--independent but, being based on semi-group estimates, it requires the result in  Theorem \ref{prop:noemb_all}. We need also   Pego and Weinstein   \cite[Theorem 3.6]{PegoWei1},  which guarantees that Theorem \ref{prop:noemb_all} allows to exclude the existence of nonzero imaginary eigenvalues also in the weighted spaces  $L^2_{a}$ in 
    \eqref{eq:pwweigh0} for $a>0$ small enough.

A further feature of the present paper is a simpler proof
of the smoothing estimates, based on direct resolvent
bounds and a Fredholm argument, without constructing
Jost solutions.

The virial--Kato methodology is proving very robust and largely model--independent,  thanks to its rather simple and natural structure. The method exploits, in fact it is based on,     a notable feature which was already emphasized in \cite{CM24D4} for the NLS,  that is the fact that near   solitons the gBBM   has better dispersion properties than near  the 0 solution. This is a key feature of the methods based on virial inequalities, both in the case of purely virial methods like \cite{KMM2017, KMM2022, KMMV2021,   Martelcubquint1, Martelcubquint2, Riallandcuqui1, Riallandcuqui2}
 and in our virial--Kato method.

 Compared to methods based purely on virial inequalities, but in common with Pego and Weinstein \cite{PegoWei2}, El Dika \cite{ElDIkaasstab} and Mizumachi \cite{MizumachiBBM},  the virial--Kato method requires something like   Theorem \ref{prop:noemb_all}.  But    there is  a     rapidly  growing list  of models for which absence of embedded eigenvalues  has been proved, for example for the NLS recently in \cite{collot2025absenceembeddedspectrumnonlinear,LiYang2026, Kim2026}; for different models, see also the other references and this paper.

   Many questions on dispersion and on more strict definitions of asymptotic stability      remain  open. We refer to Morgan \cite{morgan2025} and Chen and Zhang \cite{chen2026scatteringquinticgeneralizedbenjaminbonamahony}, as well as to Albert \cite{AlbertBBM1986,AlbertBBM1989},    for scattering near 0; we refer to Kwak and Mu\~noz
\cite{KwakMunoz2019}  for an analysis of the decay of the solutions  in a   region of spacetime different from that in Theorem \ref{thm:main}; we refer to these papers for further references.   We refer to El Dika and Martel \cite{EldikaMartel} for a  study of  the orbital stability of multi-solitons.

 \section{Notation }\label{sec:notation}

\begin{notation}\label{not:notation} We will use the following miscellanea of  notations and definitions.
\begin{enumerate}

\item In the course of the paper we will sometimes use the notation $\dot u=  \partial _t  u$ and $  u'=  \partial _x  u$.

\item  For $\Omega $ an open subset of $\R ^d$ we write $g\in C^\omega (\Omega) $ if $g$ is real analytic in $\Omega$.

\item  We will use the Japanese bracket $\<x\> :=\sqrt{1+x^2}$.


\item We set
\begin{equation}
 \gamma := (c-1) c ^{-1}\in(0,1).
 \label{eq:gamma}
\end{equation}

\item In analogy to      Kowalczyk et al.   \cite{KMM2022},      we   consider constants  $A, B,  A_1, \epsilon , \delta  >0$ satisfying
 \begin{align}\label{eq:relABg}
\log(\delta ^{-1})\gg\log(\epsilon ^{-1})\gg\log(\varepsilon ^{-1}) \gg     A  \gg     B^2\gg B    \gg 1  \text{  and  $ A_1:=B^{3/5} $.}
 \end{align}

 \item The notation    $o_{\varepsilon}(1)$  means a constant   with a parameter $\varepsilon$ such that
 $o_{\varepsilon}(1) \xrightarrow {\varepsilon  \to 0^+   }0.$

\item
Following Pego and Weinstein  \cite{PegoWei2} for $a\in \R $ we set
\begin{align}\label{eq:pwweigh0}&
   L^2_{a}:=\{ v: \ e^{ax} v\in L^2(\R )   \}  \text{ with }\|  v  \| _{L^2_{a}}:=\| e^{ax} v  \| _{L^2 (\R )}  \text{ and }\\& H^n_{a}=\{ v: \ e^{ax} v\in H^n(\R )   \}  \text{ with }\|  v  \| _{H^n_{a}}=\| e^{ax} v  \| _{H^n } \text{ for all  }n\in \N =\{ 1,2,...\}.\label{eq:pwweighn}
\end{align}

\item We also define (notice that $\widetilde{\Sigma} $ and  $\widehat{\Sigma} $ have  natural Hilbert space structure)    \begin{align}&   \nonumber
   L ^{p,s}:=\{ v: \ \< x \> ^{s}   v\in L^p(\R )   \}  \text{ with }\|  v  \| _{ L ^{p,s}}:=\| \< x \> ^{s} v  \| _{L^p(\R ) }  , \\&  \widetilde{\Sigma}:=\{ v: \ \sech \(  \sqrt{\gamma}    \kappa  x\)   v \in L^2(\R )   \}  \text{ with }\|  v \|_{ \widetilde{\Sigma} } :=\left \| \sech \(     \sqrt{\gamma} \kappa   x\)   v\right \|_{L^2(\R )}  \text{ and }  \label{eq:Lebwight}  \\&  \widehat{\Sigma}:=\{ v: \ \sech \(      \kappa  x\)   v \in L^2(\R )   \}  \text{ with }\|  v \|_{ \hat{\Sigma} } :=\left \| \sech \(     \kappa   x\)   v\right \|_{L^2(\R )}.\nonumber
\end{align}

\item Given two Banach spaces $X$ and $Y$ we denote by $\mathcal{L}(X,Y)$ the space of continuous linear operators from $X$ to $Y$. We write $\mathcal{L}(X ):=\mathcal{L}(X,X)$.

\item Following  the framework in Kowalczyk et al. \cite{KMM2022} we   fix an even function $\chi\in C_c^\infty(\R , [0,1])$ satisfying
\begin{align}  \label{eq:chi} \text{$1_{[-1,1]}\leq \chi \leq 1_{[-2,2]}$ and $x\chi'(x)\leq 0$ and set $\chi_C:=\chi(\cdot/C)$  for a  $C>0$}.
\end{align}

\item For  $ \vartheta _1=  \chi  $   and  $ \vartheta _2= 1- \chi  $  we consider the partition of unity $1= \vartheta _{1A_1} + \vartheta _{2A_1} $  where $ \vartheta _{iA_1}(x):=  \vartheta _{i }\(\frac{x}{A_1}\)$.

\item Following  the framework in Kowalczyk et al. \cite{KMM2022}  and
using the function $\chi$  in \eqref{eq:chi}
we  consider for   $C>0$  the   function
\begin{align}\label{def:zetaC}
\zeta_C(x):=\exp\(-\frac{|x|}{C}(1-\chi(x))\)   \text {  and   }  \varphi_C (x):=\int_0^x \zeta_C^2(y)\,dy.
\end{align}
We similarly define
\begin{align}\label{def:phiAi}   \varphi _{iAA_1}  (x):=\int_0^x \zeta_A^2(y) \vartheta _{iA_1} ^2 (y)\,dy  .
\end{align}

\item   We consider  semi--norms
\begin{align}
&
\| v  \|_{  { \Sigma }_{iA A_1}} :=\left \|  \left ( \vartheta _{iA_1} \zeta_A v \right ) ' \right \|_{L^2(\R )} + \left \|   \vartheta _{iA_1} \zeta_A v      \right\|_{L^2(\R )}  \text{for $i=1,2$ .} \label{eq:normA} 
\end{align}

\item We set
\begin{align}
  \label{eq:C+} {\C} _+:=\{ z\in \C : \Re z > 0     \}.
\end{align}

 \end{enumerate}

\end{notation} 

\section{Linearization and  modulation }\label{sec:mod}

The notion of modulation, based on the properties of $\ker \mathcal{L}_{[c]}$ and the generalized kernel $ \displaystyle N_g( \mathcal{L}_{[c]}):=\bigcup _{n=1}^{+\infty}\ker   \mathcal{L}_{[c]}^n    $ and pioneered by Weinstein et al. \cite{weinMod1985, weinLyap1986, PegoWei2, SW2}, is central in our theory.  Here, 
proceeding as in Miller and Weinstein \cite{MillerWeinsteinBBM} it is possible to prove the following.

\begin{lemma}\label{lem:kernel}
Let $p>1$ and $c>c^*(p)$.
The following hold, where $0<a<\kappa_c$ for an appropriately
small $\kappa_c=\kappa_c(p)>0$.
\begin{description}
    \item[a]   Let  \begin{align}& \label{eq:ker1}
  \xi _1 [c](x )   := Q _c '(x )  , \quad   \xi _2 [c](x ):= \partial _c Q _c (x ) . 
\end{align}
Then, in $L^2(\R )$,   $\ker   \mathcal{L}_{[c]} =\Span \{     \xi _1 [c]\}$ and $ \displaystyle N_g( \mathcal{L}_{[c]}) =\Span \{     \xi _1 [c] ,  \xi _2 [c]  \}   $. The same holds in $L^2_a(\R )$.  
 We have \begin{align}\label{eq:M2.3}&
   \mathcal{L}_{[c]}\xi _1 [c]=0 , \quad  \mathcal{L}_{[c]}\xi _2 [c]=-\xi _1 [c]   . 
\end{align}

 \item[b]  There are   $(c\to \theta _j [c])\in C^\omega ( (c^*(p),+\infty ), \R )   $  such that for  \small
 \begin{align}&   \eta _1 [c](x ):=\theta _1 \left ( \partial _x \partial _c Q_c- \int _{-\infty} ^{x } \partial _c Q _c (y) dy   \right ) +\theta _2(1-\partial ^2_x) Q  _c(x ) , \quad   \eta _2 [c](x ):=\theta _3   (1-\partial ^2_x)  Q _c(x )  \label{eq:ker2}
\end{align}\normalsize
such that, for $\delta_{ij}$ the Kronecker delta,
   \begin{align} & \label{eq:M2.4}
  \mathcal{L}_{[c]}^*\eta _2 [c]=0 , \quad  \mathcal{L}_{[c]}^*\eta _1 [c]=- \eta _2 [c],\\& \label{eq:M2.5}
  \<  \xi _i [c],\eta _j [c] \> = \delta_{ij}. 
\end{align}

\end{description}
\end{lemma}

\begin{proof}

Differentiating \eqref{eq:soliton_eq} with respect to $x$ and $c$ gives
\begin{align}
L_cQ_c'&=0, \label{eq:Lc-Qprime}\\
L_c\partial_cQ_c&=Q_c''-Q_c=-(1-\partial_x^2)Q_c . \label{eq:Lc-dcQ}
\end{align}
Hence the following gives  \eqref{eq:M2.3}, 
\[
\mathcal L_{[c]}Q_c'=0,
\qquad
\mathcal L_{[c]}\partial_cQ_c
=
(1-\partial_x^2)^{-1}\partial_x(-(1-\partial_x^2)Q_c)
=
-Q_c'.
\]
If $\mathcal L_{[c]}u=0$, then    $\partial_xL_cu=0.$   Thus   $L_cu=0$. 
Since $Q_c'$ has exactly one zero, the  Sturm--Liouville theory, see in particular \cite[p. 454]{CourantHilbert1}, and
\eqref{eq:Lc-Qprime} give   $\ker L_c=\Span\{Q_c'\}$
and therefore  $\ker\mathcal L_{[c]}=\Span\{\xi_1[c]\}.$
Next, we have   
\[
\mathcal L_{[c]}^*((1-\partial_x^2)Q_c)
=
-L_cQ_c'
=
0
\]
giving the first equality in \eqref{eq:M2.4} for any $\theta _3$.
Set now 
\[
F_c(x)
:=
\int_{-\infty}^x\partial_cQ_c(y)\,dy \text{ and }
G_c
:=
\partial_x\partial_cQ_c-F_c.
\]
Since from the definition of $F_c$ and $G_c$ we have
\[
(1-\partial_x^2)F_c
=
F_c-\partial_x\partial_cQ_c =-G_c,
\]
we  conclude
$(1-\partial_x^2)^{-1}G_c=-F_c,$
and therefore, by \eqref{eq:Lc-dcQ},
\[
\mathcal L_{[c]}^*G_c
=
-L_c\partial_x(1-\partial_x^2)^{-1}G_c =L_c\partial_xF_c
=
L_c\partial_cQ_c
=
-(1-\partial_x^2)Q_c,
\]
 This implies the second equality in \eqref{eq:M2.4}  if, as we will, we choose $  \theta_1(c)=\theta_3(c)$. We
     choose 
\begin{align}\label{eq:thetas}
  \theta_1(c)=\theta_3(c)=\frac1{m' (c)}  \text{  and }
\theta_2(c)=\frac{(\partial _c \|  Q_c\| _{L^1(\R )})^2}{2(m' (c))^2}, 
\end{align} for 
$m(c) :=M(Q_c)$. Notice that 
  \eqref{eq:thetas} is well defined because $m'(c)>0$ for all $1<p $ and $c> c^*(p)$,  see \cite[pp. 1165--7]{weinCPDE1987}. Notice also 
  that the formulas
\begin{align*} \|  Q_c \| _{L^2(\R) }^2 &=(c-1)^{\frac{5-p}{2(p-1)}}c^{1/2} \|  Q \| _{L^2(\R) }^2
    \text{ and }\\  \|  Q_c '\| _{L^2(\R) }^2 & =(c-1)^{\frac{p+3}{2(p-1)}}c^{-1/2}\|  Q ' \| _{L^2(\R) }^2 =(c-1)^{\frac{p+3}{2(p-1)}}c^{-1/2}  \frac{p-1}{p+3}\|  Q  \| _{L^2(\R) }^2,
\end{align*}
   where the last equality follows by testing the second equation in \eqref{eq:soliton_eq} against $Q$ and $xQ'$,
imply that the functions in \eqref{eq:thetas} are   in $ C^\omega ( (c^*(p),+\infty ), \R )$.
Then
\begin{align*}
   \langle Q_c',(1-\partial_x^2)Q_c\rangle=0 \text{ and } \langle \partial_cQ_c,(1-\partial_x^2)Q_c\rangle=m'(c)\Longrightarrow   
  \<  \xi _i [c],\eta _2 [c] \> =\delta _{i2}.
\end{align*}
We have
\begin{align*}
\langle Q_c',G_c\rangle
&=
-\langle Q_c',(1-\partial_x^2)F_c\rangle
=
 \langle (1-\partial_x^2)Q_c ,F_c'\rangle =\langle (1-\partial_x^2)Q_c ,\partial _cQ_c\rangle =m'(c)
\end{align*}
and
\begin{align*}
\langle\partial_cQ_c,G_c\rangle
&=
\langle \partial_cQ_c ,\partial_x\partial_cQ_c\rangle  - \langle F_c', F_c\rangle = - \langle F_c', F_c\rangle =
-\frac12
\left(
\int_{\R}\partial_cQ_c\,dx
\right)^2
.
\end{align*}
    Therefore we obtain the following, completing the proof of \eqref{eq:M2.5},
\[
\langle\xi_1,\eta_1\rangle=1 
 \text{ and }
\langle\xi_2,\eta_1\rangle
=
-\frac{(\partial _c \|  Q_c\| _{L^1(\R )})^2}{2m' (c)}
+
\frac{(\partial _c \|  Q_c\| _{L^1(\R )})^2}{2m' (c)}
=
0.
\]
Since $\mathcal L_{[c]}\xi_2=-\xi_1,$ we have $\Span\{\xi_1,\xi_2\}
\subseteq N_g(\mathcal L_{[c]}).$
If $\zeta\in\ker\mathcal L_{[c]}^2$, then
 $\mathcal L_{[c]}\zeta=\alpha\xi_1,$
 and hence
 $\mathcal L_{[c]}(\zeta+\alpha\xi_2)=0.$
 Thus
 $\zeta+\alpha\xi_2=\beta\xi_1,$
so
\[
\ker\mathcal L_{[c]}^2
=
\Span\{\xi_1,\xi_2\}.
\]
If a Jordan chain of length three existed, there would be $\zeta$ and
$\alpha,\beta$, with $\beta\neq0$, such that
\[
\mathcal L_{[c]}\zeta
=
\alpha\xi_1+\beta\xi_2.
\]
Applying  $\langle \cdot , \eta_2\rangle $ gives
\[
0
=
\langle\zeta,\mathcal L_{[c]}^*\eta_2\rangle
=
\langle\mathcal L_{[c]}\zeta,\eta_2\rangle
=
\beta,
\]
a contradiction. Consequently, in $L^2(\R )$
\[
N_g(\mathcal L_{[c]})
=
\Span\{\xi_1[c],\xi_2[c]\}.
\]
The same argument applies in $L_a^2(\R)$ for
$0<a<\kappa_c$.  See also \cite{PegoWei2} and \cite{MillerWeinsteinBBM}.
\end{proof}

Key for us is the following absence of nonzero embedded eigenvalues.  
 
\begin{proposition}\label{prop:noembeig1} For any $p>1$ and $c>c^*(p)$ the operator $\mathcal{L}_{[c]}$ does not have in $L^2(\R)$   any nonzero eigenvalue $\lambda$  with $\Re \lambda \ge 0$. The same statement is true in $L^2_a(\R)$ for $0<a<\kappa _c$.
    \end{proposition}
\begin{proof}  The  index theorem in
\cite[Theorem 11.1]{LZ22Memoiers} gives
\begin{equation}
k_r+2k_c+2k_i^{\leq 0}+k_0^{\leq 0}
=n^-(L_c) .
\end{equation}
 On the right we have the total multiplicity of the negative eigenvalues of $L_c$, which is $n^-(L_c)=1$. Since $m'(c)>0$, see below \eqref{eq:thetas}, $k_0^{\leq 0}=1$ by 
\cite[Theorem 11.1]{LZ22Memoiers}. This implies that $k_r=k_c=k_i^{\leq 0}=0$  where $k_r$ is the sum of the multiplicities of the strictly positive eigenvalues of 
$\mathcal{L}_{[c]}$ and $k_c$ is the sum of the multiplicities of the  eigenvalues  of 
$\mathcal{L}_{[c]}$ in the interior of the first quadrant in $\C$. This implies in particular that there are no eigenvalues with
  $\Re\lambda>0$  in
$L^2(\R)$. 
The same conclusion holds by simple arguments in $L_a^2(\R)$ for
$0<a<\gamma ^{1/2}\kappa_c$, with $\kappa_c$    the constant in {Lemma} \ref{lem:roots1}.  

The absence of $0\neq\lambda\in\im\R$ eigenvalues in $L^2(\R)$
is because of Theorem \ref  {prop:noemb_all},    proved in \S\ref{sec:noemb}.
Finally, this implies the absence of
$0\neq\lambda\in\im\R$ eigenvalues in $L_a^2(\R)$ by the argument
developed in detail for KdV--type equations in Pego and Weinstein
\cite[Theorem~3.6]{PegoWei1}, which, as stated therein,  applies also to BBM.
 \end{proof}

For $c_0>1$ and $r\in\R_+:=(0,\infty)$ and for   $D _{H ^1  (\R )}(u,r):=\{v\in H ^1  (\R )\ |\ \|u-v\| _{H ^1  (\R )}<r \}$,  we set     $$\mathcal{U} (c _0 ,r  ) := \bigcup _{  x_0\in \R }  e^{-x_0 \partial _x}  D_{H ^1  (\R )}({Q}_{c _0},r ). $$

\begin{lemma}[Modulation]\label{lem:mod1}   For any $c_0>c^*(p)$
there exist     $\epsilon _0 >0$  and  $ B_0>0$   with $\epsilon  _0 \ll B_0 ^{-3/2}$
such that for $  \epsilon  _0 ^{-1 }\gg    B\ge B_0$ and for the function  $\zeta _{B}$   in \eqref{def:zetaC} there exists
$$(c,D)     \in C^1\left (  \mathcal{U} (c_0 ,\epsilon  _0 ), ( c^*(p), +\infty ) \times \R   \right ), $$
 such that for  any $ {u} \in \mathcal{U} (c_0,\epsilon  _0   )$
  \begin{align}\label{61} &   {v} ( {u}):=e^{   D( {u}) \partial _x }  {u} -  {Q}  _{c (u)}     \text{ satisfies }   \\& \nonumber  \<   {v} ( {u}), \zeta _{B} \eta _1 [c (u)]   \> =  \<    {v} ( {u}),   \eta _2 [c (u)]   \> =0. \end{align}  
We have the identities
\begin{align}& c  (  {Q}   _{c_1}  )  =c_1  \text{, }   D  ( {Q}   _{c_1}  )=0
  \text{, }   D  (  e^{  - D _0 \partial _x }   {u})    = D  (    { u})    + D_0 \text{ and } c  (  e^{  - D _0 \partial _x }   {u})    =  c  (     {u}).\label{eq:equivariance}
\end{align}
Furthermore,  there exists a constant $K  >0$ such that
\begin{align}
  \label{eq:mod2}  B ^{-\frac{1}2}|D( {u}) -D_0| + |c ({u})- c_0|\le K      \epsilon  _0 \text{ for all  }  {u}\in  e^{-D_0 \partial _x}  D_{ {H} ^1   }( {Q}   _{c_0}    ,\epsilon   _0    ).
\end{align}
Set
$r_0:=(c_0-c^*(p))/2$ and
$J_0:=[c_0-r_0,c_0+r_0]\subset(c^*(p),+\infty)$.
After decreasing $\epsilon_0$ if necessary, we assume
$K\epsilon_0\leq r_0$.
Then \eqref{eq:mod2} implies that
$c(u)\in J_0$ for every $u\in\mathcal U(c_0,\epsilon_0)$.
\end{lemma}
  \begin{proof}
The proof follows from a standard application of the implicit function theorem and is the same as the proof of the analogous result  in \cite{CM25D1}.

\end{proof}

Notice that since $(u\to  \partial_x  R_{-\partial _x ^2} (-1)f(u) ) $ is a locally Lipschitz map in $H^1(\R )$, equation \eqref {eq:BBM2}  is a classical ODE in $H^1(\R )$ and that, in particular, solutions satisfy $   u\in C^1\left (   \R , H^1(\R )\right )$. By the orbital stability, we can assume    $u(t)\in \mathcal{U} (c _0 , \epsilon  ) $ 
for all times $t\in \R$. 
Then the two $C^1$ functions
$(c(t),D(t))=(c(u(t)),D(u(t)))$ remain defined,
with $c(t)\in J_0$ for all $t\in\R$.
Inserting in \eqref{eq:BBM1} the ansatz 
$u=Q_c(x-D)+v(x-D)$ after cancellations we obtain 
\small
\begin{align}      &
    \dot v   =\mathcal L_{[c]}v  - R_{-\partial _x ^2} (-1) \partial _x    \left ( f( Q   _{c} + v  )  -  f( Q  _{c} )-  f'(Q  _{c} ) v\right )  
    +   (\dot D -c )   (   \xi _1[c]+   v' )-    \dot c  \xi _2[c]              \label{eq:gKdV3}
\end{align}
\normalsize
or also    
\begin{align}      &
    \dot v   =\mathcal L^{(0)}_{[c]}v  - R_{-\partial _x ^2} (-1) \partial _x    \left ( f( Q   _{c} + v  )  -  f( Q  _{c} ) \right )  
    +   (\dot D -c )     (   \xi _1[c]+   v' )-    \dot c   \xi _2[c]                \label{eq:gKdV3fl}
\end{align}
where, for the multiplication operator $v\to  f'(Q  _{c} )v$,  
\begin{equation}
\mathcal{L}_{[c]}^{(0)} =R_{-\partial _x ^2} (-1)\partial_x (-c\partial_x^2+c-1),
  \text{ and  }
\mathcal L_{[c]} =\mathcal{L}_{[c]}^{(0)}   - R_{-\partial _x ^2} (-1)\partial_x \circ   f'(Q  _{c} )  .
 \label{eq:Lin0}
\end{equation}

\section{Main estimates}\label{sec:mainest}

The proof of the main result is mainly based on the following continuation argument.

\begin{proposition}\label{prop:continuation}
For any small $ \varepsilon >0$ and for large constants $A$, $B$  and  $A_1:= B ^{3/5}$ as in formula \eqref{eq:relABg}
there exists  a    $\epsilon _0= \epsilon _0(c_0, \varepsilon )   $ s.t.\  if
\begin{align}
  \label{eq:main2}
\|v \| _{L^2(I,  { \Sigma }_{1A A_1} )} +
\|v \| _{L^2(I,  { \Sigma }_{2A A_1} )} +  \|v \| _{L^2(I, \widetilde{\Sigma}  )} + \| \dot D - c \| _{L^2(I  )} + \| \dot c \| _{L^2(I  )}\le \varepsilon
\end{align}
holds  for $I=[0,T]$ for some $T>0$ and if  the constant $\epsilon$ in \eqref{eq:orbstab} satisfies  $\epsilon  \in (0, \epsilon _0)$,
then in fact for $I=[0,T]$    inequality   \eqref{eq:main2} holds   for   $\varepsilon$ replaced by $   o _{B ^{-1} }(1) \varepsilon $.
\end{proposition}
\begin{proof}
    The proof follows, by exactly the same argument as in   \cite{CM25D1} from the next three lemmas.
\end{proof}

\begin{lemma}\label{lem:lemdscrt} We have the estimates
 \begin{align} \label{eq:discrest1}
   & |\dot c    |   +   |\dot D -c |      \lesssim  B ^{-1}    \left (  \|v \| _{    \Sigma  _{1A A_1}  } +
\|v \| _{     \Sigma  _{2A A_1} } \right )   .
\end{align}

\end{lemma}
\begin{proof}
    The proof is almost verbatim the same as the analogous one in \cite{CM25D1} and is skipped. 
\end{proof}

\begin{lemma}[Virial Inequality]\label{prop:1virial}
There exists a fixed constant $C_{c_0}>0$ such that
\begin{align}
\label{eq:sec:1virial11}
\|v\|_{L^2(I,\Sigma_{1AA_1})}
&\leq
C_{c_0}
\left(
\epsilon ^{1/3}
+
\|v\|_{L^2(I,\widetilde{\Sigma})}
+
A_1^{-1}
\|v\|_{L^2(I,\Sigma_{2AA_1})}
\right),
\\
\label{eq:sec:1virial12}
\|v\|_{L^2(I,\Sigma_{2AA_1})}
&\leq
C_{c_0}
\left(
\epsilon ^{1/3}
+
A_1^{-1}\|v\|_{L^2(I,\widetilde{\Sigma})}
+
A_1^{-1}
\|v\|_{L^2(I,\Sigma_{1AA_1})}
\right).
\end{align}
\end{lemma}
We prove Lemma \ref{prop:1virial} in Sec. \ref{sec:virial}. Finally, the following follows from the material from  \S \ref{sec:disp} on.

\begin{lemma}[Smoothing Inequality]\label{prop:smooth11}
We have
\begin{align}
\label{eq:sec:smooth11}
\|v\|_{L^2(I,\widetilde{\Sigma})}
\leq
o_{B^{-1 }}(1)
\left(
\|v\|_{L^2(I,\Sigma_{1AA_1})}
+
\|v\|_{L^2(I,\Sigma_{2AA_1})}
\right)
+
o_{B^{-1}}(1)\varepsilon .
\end{align}
\end{lemma}

\begin{proof}[Proof of Theorem \ref{thm:main}] The inequality \eqref{eq:asstab2} is a simple consequence of Proposition \ref{prop:continuation},
while \eqref{eq:asstab1} follows from the modulation Lemma \ref{lem:mod1}.  Consider   now 
 \begin{equation}
  \begin{aligned}\label{eq:bolda}
   \mathbf{a}(t) &:= 2 ^{-1} \langle  e^{- 2a\< x\>}   v (t), v (t)  \rangle    .
  \end{aligned}
  \end{equation}
Then  \small
\begin{align*}
\dot{\mathbf a}  &=  \langle  e^{- 2a\< x\>}   v  , \mathcal L^{(0)}_{[c]}v  - R_{-\partial _x ^2} (-1) \partial _x    \left ( f( Q   _{c} + v  )  -  f( Q  _{c} ) \right )  
    +   (\dot D -c )     (   \xi _1[c]+   v' )-    \dot c   \xi _2[c]   \rangle \\&  = O(\| v \| _{H^1}^2)     = O(B  \epsilon ^2) .
\end{align*}
\normalsize
Thus $t\to \mathbf{a}(t)$ is  globally Lipschitz. Since $\mathbf a\in L^1(\R_+)$ by \eqref{eq:asstab2},  we have $\mathbf{a}(t)\xrightarrow{t\to +\infty}0 $, proving   \eqref{eq:asstab20}.   
Next we prove \eqref{eq:asstab3}.
Since $c(t)\in J_0$ and
$
m_*:=\min_{c\in J_0}m'(c)>0,
$
it suffices to show that $m(c(t))$ converges as $t\to+\infty$.
Since 
\begin{align*}
    M(u_0) =m(c)+M(v) + \langle v, (1-\partial _x ^2) Q_c \rangle = m(c)+M(v) 
\end{align*}
by $\langle v,\eta_2 [c] \rangle =0$ and by $\eta_2 [c] =\theta _3[c] (1-\partial _x ^2) Q_c$,   it suffices   to prove $\frac{d}{dt}M( {v}) \in L^1(\R _+) $. We have \small
\begin{align*}
\frac{d}{dt}M( {v})&=    \<  (1   -\partial _x ^2) v  ,  \dot v  \> = I+II \text{ with }\\
I&:= (\dot D -c ) \< v , (1   -\partial _x ^2) \xi _1[c]\> -  \dot c \<v ,  (1   -\partial _x ^2) \xi _2[c]\>   \text{ and}\\  II&:=  \< v ',    f  ( v  +Q_c) -f  (  Q_c)      \> 
   ,
\end{align*}  \normalsize
after some cancellations.  By the exponential decay of $\xi_1[c]$ and $\xi_2[c]$,
\begin{align*}
|I(t)|
\lesssim (|\dot D-c|+|\dot c|)
\bigl(
\|  v(t)\|_{\Sigma_{1AA_1}}
+\|  v(t)\|_{\Sigma_{2AA_1}}
\bigr).
\end{align*}
Thus \eqref{eq:main2} on $\R_+$ implies $I\in L^1(\R_+)$.    Next we observe 
that for $1<p\le2$,  we have $|f'(r)-f'(s)|\le C_p |r-s|^{p-1}$
 for $ r,s\in\R$, hence, like in \cite{BaCuDeSaMa}, 
\begin{align*}
II =\< v ',    f  ( v  +Q_c) -f  (  Q_c) -f(v)     \> 
=\int_0^1 \< v '  v ,    f ' ( s v  +Q_c) -f '  (sv)     \>  ds .
\end{align*}
For $A$ sufficiently large,  
  $Q_c^{p-1}\lesssim\zeta_A^2$  uniformly for $c$ near $c_0$
 gives
\begin{align*}
|II |
\lesssim\int_{\R}Q_c^{p-1}|v v '|\,dx
\lesssim\|\zeta_Av (t)\|_{H^1}^2
\lesssim
\|  v(t)\|_{\Sigma_{1AA_1}}^2
+\|  v(t)\|_{\Sigma_{2AA_1}}^2.
\end{align*}
By \eqref{eq:main2}, the last expression belongs to $L^1(\R_+)$.
This proves $II\in L^1(\R_+)$ and completes the proof of $\frac{d}{dt}M( {v}) \in L^1(\R _+) $ and of the last statement in Theorem \ref{thm:main}.

\end{proof}

\section{The virial inequality: proof of Lemma \ref{prop:1virial}}
\label{sec:virial}

Using the notation in \eqref{def:phiAi}, we set
\begin{align}
\label{def:funct:vir}
\mathcal I_i
:=
\int_{\R}
\varphi_{iAA_1}(x)
\left(v^2+(v')^2\right)\,dx,
\qquad i=1,2.
\end{align}

\begin{lemma}\label{lem:1stV1}
There exist fixed constants $C_j=C_j(c_0)>0$ such that
\begin{align}
\label{eq:lem:1stV11}
\|v\|_{\Sigma_{1AA_1}}^2
&\leq
-C_1\frac{d}{dt}\mathcal I_1
-C_2\frac{d}{dt}\mathcal I_2
+
C_3
\|v\|_{\widetilde{\Sigma}}^2
,
\\
\label{eq:lem:1stV12}
\|v\|_{\Sigma_{2AA_1}}^2
&\leq
-C_4A_1^{-2} \frac{d}{dt}\mathcal I_1-C_5\frac{d}{dt}\mathcal I_2
+
C_6  
A_1^{-2}\|v\|_{\widetilde{\Sigma}}^2
.
\end{align}
\end{lemma}

\begin{proof}
Using \eqref{eq:gKdV3fl}, we obtain
\begin{align*}
\frac{d}{dt}\mathcal I_i
= K_i^{(0)}
+
K_i^f
+
K_i^D
+
K_i^c 
\end{align*}
where \small
\begin{align}
K_i^{(0)} &:=2\langle
\varphi_{iAA_1}v,\mathcal L_{[c]}^{(0)}v\rangle
+
2\langle
\varphi_{iAA_1}v' ,
\left(
\mathcal L_{[c]}^{(0)}v
\right)'\rangle , 
\nonumber\\
K_i^f:=&-
2\langle
\varphi_{iAA_1}v,
R_{-\partial_x^2}(-1)\partial_x
\left(
f(Q_c+v)-f(Q_c)
\right)\rangle  -
2\langle 
\varphi_{iAA_1}v', 
\partial_xR_{-\partial_x^2}(-1)\partial_x
\left(
f(Q_c+v)-f(Q_c)
\right)\rangle
\nonumber\\
K_i^D &:=  
2(\dot D-c)
\int_{\R}
\varphi_{iAA_1}v
\left(
\xi_1[c]+v'
\right)\,dx
 +
2(\dot D-c)
\int_{\R}
\varphi_{iAA_1}v'
\left(
\xi_1[c]+v'
\right) '  dx
\nonumber\\
K_i^c&:= -
2\dot c
\int_{\R}
\varphi_{iAA_1}
\left(
v\xi_2[c]+v'\xi_2'[c]
\right) dx
 .
\label{eq:virial-decomposition}
\end{align}
\normalsize
We first consider $K_i^{(0)}$. By \eqref{eq:Lin0} we have $\mathcal L_{[c]}^{(0)} 
=
c\partial _x
-
R_{-\partial_x^2}(-1)\partial _x$.  
Hence
\begin{align}
K_i^{(0)}
={}&
2c\int_{\R}
\varphi_{iAA_1}
\left(
vv'+v'v''
\right)\,dx
-
2\int_{\R}
\varphi_{iAA_1}
\left[
vR_{-\partial_x^2}(-1)v'
+
v'
\partial_xR_{-\partial_x^2}(-1)v'
\right]\,dx .
\label{eq:K0-first}
\end{align}
The first term in the right hand side in \eqref{eq:K0-first} satisfies
\begin{align}
2c\int_{\R}
\varphi_{iAA_1}
\left(
vv'+v'v''
\right)\,dx
=
-c\int_{\R}
\varphi_{iAA_1}'
\left(
v^2+(v')^2
\right)\,dx.
\label{eq:K0-transport}
\end{align}
Since  $\partial_x^2R_{-\partial_x^2}(-1)
=
R_{-\partial_x^2}(-1)-1,$
 the second integral in the right hand side in \eqref{eq:K0-first} is equal to
\begin{align*}
&
-2\int_{\R}
\varphi_{iAA_1}
\left[
v\partial_xR_{-\partial_x^2}(-1)v
+
v'
\left(
R_{-\partial_x^2}(-1)v-v
\right)
\right]\,dx
\\
&=
-2\int_{\R}
\varphi_{iAA_1}
\partial_x
\left[
vR_{-\partial_x^2}(-1)v
\right]\,dx
+
2\int_{\R}
\varphi_{iAA_1}vv'\,dx
\\
&=
2\int_{\R}
\varphi_{iAA_1}'
vR_{-\partial_x^2}(-1)v\,dx
-
\int_{\R}
\varphi_{iAA_1}'v^2\,dx.
\end{align*}
Using \eqref{def:phiAi}, it follows that
\begin{align}
K_i^{(0)}
={}&
-c\int_{\R}
\zeta_A^2\vartheta_{iA_1}^2
\left(
v^2+(v')^2
\right)\,dx
-
\int_{\R}
\zeta_A^2\vartheta_{iA_1}^2v^2\,dx+
2\int_{\R}
\zeta_A^2\vartheta_{iA_1}^2
vR_{-\partial_x^2}(-1)v\,dx.
\label{eq:K0-exact}
\end{align}
We have 
\begin{align}
    \label{eq:intker}R_{-\partial_x^2}(-1)g(x)=  2^{-1} \int _\R e^{-|x-y|} g(y)dy.
\end{align}
  Moreover,
\begin{align}
\left|
\left(
\zeta_A\vartheta_{iA_1}
\right)'(x)
\right|
\lesssim
\left(
A^{-1}+A_1^{-1}
\right)
\zeta_A(x).
\label{eq:weight-derivative}
\end{align}
We claim
\begin{align}
&
\left\|
\zeta_A\vartheta_{iA_1}
R_{-\partial_x^2}(-1)v
-
R_{-\partial_x^2}(-1)
\left(
\zeta_A\vartheta_{iA_1}v
\right)
\right\|_{L^2}
\lesssim
\left(
A^{-1}+A_1^{-1}
\right)
\left(
\|v\|_{\Sigma_{1AA_1}}
+
\|v\|_{\Sigma_{2AA_1}}
\right).
\label{eq:resolvent-commutator}
\end{align}
Assuming \eqref{eq:resolvent-commutator}   and using $\|
R_{-\partial_x^2}(-1)
\| _{\mathcal L(L^2(\R))}\le 1$
we obtain
\begin{align*}
\int_{\R}
\zeta_A^2\vartheta_{iA_1}^2
vR_{-\partial_x^2}(-1)v\,dx
&=
\left\langle
\zeta_A\vartheta_{iA_1}v,
R_{-\partial_x^2}(-1)
\left(
\zeta_A\vartheta_{iA_1}v
\right)
\right\rangle
\\
&\quad+
\left\langle
\zeta_A\vartheta_{iA_1}v,
\zeta_A\vartheta_{iA_1}
R_{-\partial_x^2}(-1)v
-
R_{-\partial_x^2}(-1)
\left(
\zeta_A\vartheta_{iA_1}v
\right)
\right\rangle
\\
&\leq
\left\|
\zeta_A\vartheta_{iA_1}v
\right\|_{L^2}^2
\\
&\quad+
C
\left(
A^{-1}+A_1^{-1}
\right)
\left\|
\zeta_A\vartheta_{iA_1}v
\right\|_{L^2}
\left(
\|v\|_{\Sigma_{1AA_1}}
+
\|v\|_{\Sigma_{2AA_1}}
\right).
\end{align*}
Thus,   by \eqref{eq:relABg} and  by  Young's inequality  in the last term we conclude that for every fixed $\varepsilon_1>0$ 
\begin{align}
&
2\int_{\R}
\zeta_A^2\vartheta_{iA_1}^2
vR_{-\partial_x^2}(-1)v\,dx
 \leq
(2+\varepsilon_1)
\left\|
\zeta_A\vartheta_{iA_1}v
\right\|_{L^2}^2
+
C_{\varepsilon_1}A_1^{-2}
\left(
\|v\|_{\Sigma_{1AA_1}}^2
+
\|v\|_{\Sigma_{2AA_1}}^2
\right).
\label{eq:resolvent-free-estimate}
\end{align} 
From \eqref{eq:K0-exact} we conclude
\begin{align}
K_i^{(0)}
\leq{}&
-(c-1-\varepsilon_1)
\left\|
\zeta_A\vartheta_{iA_1}v
\right\|_{L^2}^2
-
c
\left\|
\zeta_A\vartheta_{iA_1}v'
\right\|_{L^2}^2
+
C_{\varepsilon_1}A_1^{-2}
\left(
\|v\|_{\Sigma_{1AA_1}}^2
+
\|v\|_{\Sigma_{2AA_1}}^2
\right).
\label{eq:K0-coercive1}
\end{align}
Furthermore, $\left(
\zeta_A\vartheta_{iA_1}v
\right)'
=
\zeta_A\vartheta_{iA_1}v'
+
\left(
\zeta_A\vartheta_{iA_1}
\right)'v$ 
and hence
\begin{align}
\left\|
\left(
\zeta_A\vartheta_{iA_1}v
\right)'
\right\|_{L^2}^2
\lesssim
\left\|
\zeta_A\vartheta_{iA_1}v'
\right\|_{L^2}^2
+
A_1^{-2}
\left(
\|v\|_{\Sigma_{1AA_1}}^2
+
\|v\|_{\Sigma_{2AA_1}}^2
\right).
\label{eq:weighted-derivative-estimate}
\end{align}
Since $c(t)$ remains in a sufficiently small fixed neighborhood of
$c_0>1$, taking $\varepsilon_1>0$ sufficiently small gives
$\nu_0=\nu_0(c_0)>0$ such that
\begin{align}
K_i^{(0)}
\leq
-\nu_0\|v\|_{\Sigma_{iAA_1}}^2
+
CA_1^{-2}
\left(
\|v\|_{\Sigma_{1AA_1}}^2
+
\|v\|_{\Sigma_{2AA_1}}^2
\right).
\label{eq:K0-coercive2}
\end{align}
Now we prove \eqref {eq:resolvent-commutator}. On the left hand side we have an  operator whose integral kernel 
  can be majorized by 
\begin{align*}
 2^{-1}e^{-|x-y|}
\left |
\zeta_A(x)\vartheta_{iA_1}(x)
-
\zeta_A(y)\vartheta_{iA_1}(y)
\right | &\lesssim  \left(
A^{-1}+A_1^{-1}
\right)  e^{-|x-y|} |x-y|  \zeta_A(c_{x,y}) \\&\lesssim   \left(
A^{-1}+A_1^{-1}
\right)  e^{-|x-y|/2}    \zeta_A(y)   ,
\end{align*}
for $c_{x,y}$ between  $x$ and $y$. By Young's convolution inequality \cite[Theorem 4.33]{brezisbook} and the partition of unity $1= \vartheta _{1A_1} + \vartheta _{2A_1} $,  \eqref {eq:resolvent-commutator} ensues.

We next consider $K_i^f$. Using again
$\partial_x^2R_{-\partial_x^2}(-1)
=
R_{-\partial_x^2}(-1)-1$, we obtain
\begin{align*}
K_i^f
={}&
-2\int_{\R}
\varphi_{iAA_1}v
\partial_xR_{-\partial_x^2}(-1)
\left(
f(Q_c+v)-f(Q_c)
\right)\,dx
\\
&-
2\int_{\R}
\varphi_{iAA_1}v'
( R_{-\partial_x^2}(-1) -1) 
\left(
f(Q_c+v)-f(Q_c)
\right) \,dx
\\
={}&
2\int_{\R}
\varphi_{iAA_1}'
vR_{-\partial_x^2}(-1)
\left(
f(Q_c+v)-f(Q_c)
\right)\,dx
 +
2\int_{\R}
\varphi_{iAA_1}v'
\left(
f(Q_c+v)-f(Q_c)
\right)\,dx.
\end{align*}
The last integral can be integrated by parts without differentiating
$f'$ since \small
\begin{align*}
&
v'
\left(
f(Q_c+v)-f(Q_c)
\right)
=
\partial_x
\left ( F(Q_c+v)-F(Q_c )
-
f(Q_c)v
\right )
-
Q_c'
\left(
f(Q_c+v)-f(Q_c)-f'(Q_c)v
\right).
\end{align*}
Consequently,
\begin{align}
K_i^f
={}&
2\int_{\R}
\varphi_{iAA_1}'
vR_{-\partial_x^2}(-1)
\left(
f(Q_c+v)-f(Q_c)
\right)\,dx
\nonumber\\
&-
2\int_{\R}
\varphi_{iAA_1}'
\left (
 F(Q_c+v)-F(Q_c )
-
f(Q_c)v
\right ) dx
\label{eq:Kf-exact} \\
&-
2\int_{\R}
\varphi_{iAA_1}Q_c'
\left(
f(Q_c+v)-f(Q_c)-f'(Q_c)v
\right)\,dx.
\nonumber
\end{align} \normalsize
We write \small
\begin{align}
f(Q_c+v)-f(Q_c)
&=
f'(Q_c)v
+
\left(
f(Q_c+v)-f(Q_c)-f'(Q_c)v
\right), \label{eq:exp1}
\\
F(Q_c+v)-F(Q_c)-f(Q_c)v
&=
\frac12 f'(Q_c)v^2
+ (
F(Q_c+v)-F(Q_c)-f(Q_c)v-\frac12 f'(Q_c)v^2)
.  \label{eq:exp2}
\end{align} \normalsize
Hence the   contribution to \eqref{eq:Kf-exact} of the first term on the right in each of \eqref{eq:exp1}--\eqref{eq:exp2}   
 is
\begin{align*}
&
2\int_{\R}\varphi_{1AA_1}'v
R_{-\partial_x^2}(-1)\left(f'(Q_c)v\right)\,dx
-
\int_{\R}\varphi_{1AA_1}'f'(Q_c)v^2\,dx=O(\|v\|_{\widetilde{\Sigma}}^2)     .
\end{align*} 
We next estimate the nonlinear remainders. Since $1<p<2$ 
\begin{align*}
\left|
f(Q_c+v)-f(Q_c)-f'(Q_c)v
\right|
&\lesssim |v|^p,
\\
\left|
F(Q_c+v)-F(Q_c)-f(Q_c)v-\frac12 f'(Q_c)v^2
\right|
&\lesssim |v|^{p+1}.
\end{align*}
Therefore, again by \eqref{eq:resolvent-commutator} and $\|v\|_{L^\infty}\lesssim B^{ 1/2}\epsilon$,
\begin{align*}
&
\left|
2\int_{\R}\varphi_{1AA_1}'v
R_{-\partial_x^2}(-1)
\left(
f(Q_c+v)-f(Q_c)-f'(Q_c)v
\right)\,dx
\right|
\\
&\quad+
\left|
2\int_{\R}\varphi_{1AA_1}'
\left(
F(Q_c+v)-F(Q_c)-f(Q_c)v-\frac12f'(Q_c)v^2
\right)\,dx
\right|
\\
&\lesssim
B^{\frac{p-1}{2}} \epsilon ^{ {p-1} }
\left(
\|v\|_{\Sigma_{1AA_1}}^2
+
\|v\|_{\Sigma_{2AA_1}}^2
\right).
\end{align*}
For the  term in the last line  in \eqref{eq:Kf-exact}, set
\begin{align*}
\Omega_{1,t}
:=
\left\{
x\in\R:\ 2|v(t,x)|\leq Q_c(x)
\right\},
\qquad
\Omega_{2,t}
:=
\R\setminus\Omega_{1,t}.
\end{align*}
On $\Omega_{1,t}$,  we have 
\begin{align*}
\left|
f(Q_c+v)-f(Q_c)-f'(Q_c)v
\right|
\lesssim
Q_c^{p-2}v^2,
\end{align*}
and, since $|Q_c'|\lesssim Q_c$,
\begin{align*}
|Q_c'|
\left|
f(Q_c+v)-f(Q_c)-f'(Q_c)v
\right|
\lesssim
Q_c^{p-1}v^2.
\end{align*}
On $\Omega_{2,t}$ we have $Q_c\lesssim |v|$, and therefore
\begin{align*}
|Q_c'|
\left|
f(Q_c+v)-f(Q_c)-f'(Q_c)v
\right|
&\lesssim
Q_c|v|^p
\lesssim
Q_c^{\frac{p-1}{2}}
|v|^{\frac{p-1}{2}+2}.
\end{align*}
Using $\|v\|_{L^\infty}\lesssim B^{ 1/2}\epsilon$ and the exponential
decay of $Q_c$, we conclude that
\begin{align*}
&
\left|
2\int_{\R}
\varphi_{1AA_1}Q_c'
\left(
f(Q_c+v)-f(Q_c)-f'(Q_c)v
\right)\,dx
\right|
\\
&\lesssim
 B^{\frac{p-1}{4}} \epsilon ^{\frac{p-1}{2}}
 (
\|v\|_{\Sigma_{1AA_1}}^2
+
\|v\|_{\Sigma_{2AA_1}}^2)
+
\|v\|_{\widetilde{\Sigma}}^2
.
\end{align*}
Combining the preceding estimates, we obtain
\begin{align}
|K_1^f|
\lesssim 
\|v\|_{\widetilde{\Sigma}}^2
+
A_1^{-2}
\left(
\|v\|_{\Sigma_{1AA_1}}^2
+
\|v\|_{\Sigma_{2AA_1}}^2
\right).
\label{eq:Kf1}
\end{align}
For $i=2$, the support properties of $\vartheta_{2A_1}$, together
with \eqref{def:phiAi} and the exponential decay of $Q_c$, yield
\begin{align}
|K_2^f|
\lesssim
A_1^{-100}
 (
\|v\|_{\Sigma_{1AA_1}}^2
+
\|v\|_{\Sigma_{2AA_1}}^2
+
\|v\|_{\widetilde{\Sigma}}^2
 )
 \lesssim  A_1^{-100}
 (
\|v\|_{\Sigma_{1AA_1}}^2
+
\|v\|_{\Sigma_{2AA_1}}^2 
 ) .
\label{eq:Kf2}
\end{align}
We next estimate $K_i^D$. Splitting the terms containing $\xi_1[c]$
from those containing $v'$, we have \small
\begin{align}
K_i^D
={}&
2(\dot D-c)
\int_{\R}
\varphi_{iAA_1}
\left(
v\xi_1[c]+v'\xi_1'[c]
\right)\,dx
+
2(\dot D-c)
\int_{\R}
\varphi_{iAA_1}
\left(
vv'+v'v''
\right)\,dx.
\label{eq:KD-first}
\end{align}\normalsize
Integration by parts, first for $v\in C_0^2 (\R)$ with bounds extended by density,   gives
$$
\left|2\int_{\R}\varphi_{iAA_1}
\left(vv'+v'v''\right)\,dx\right|=
\int_{\R}\varphi_{iAA_1}'\left(v^2+v'^2\right),dx
\lesssim
\|v\|_{\Sigma_{iAA_1}}^2
+A_1^{-2}\left(
\|v\|_{\Sigma_{1AA_1}}^2+\|v\|_{\Sigma_{2AA_1}}^2
\right).
$$
Since $\xi_1[c]$ and $\xi_1'[c]$ decrease exponentially, we have
$$
\left|
\int_{\R}\varphi_{1AA_1}
\left(
v\xi_1[c]+v'\xi_1'[c]
\right)\,dx
\right|
\lesssim
\|v\|_{\Sigma_{1AA_1}}
+
A_1^{-100}\|v\|_{\Sigma_{2AA_1}}.
$$
Similarly, since
$\supp\varphi_{2AA_1}\subset\{|x|\gtrsim A_1\}$,
$$
\left|
\int_{\R}\varphi_{2AA_1}
\left(
v\xi_1[c]+v'\xi_1'[c]
\right)\,dx
\right|
\lesssim
A_1^{-100}
\left(
\|v\|_{\Sigma_{1AA_1}}
+
\|v\|_{\Sigma_{2AA_1}}
\right).
$$
Then, by \eqref{eq:relABg} and orbital stability,
\[
\|v\|_{\Sigma_{1AA_1}}+\|v\|_{\Sigma_{2AA_1}}
\lesssim \|v\|_{H^1(\R)}
\lesssim B^{1/2}\epsilon
\ll A_1^{-100}.
\]
Together with \eqref{eq:discrest1} and the preceding estimates,
this gives
\begin{align} \nonumber
|K_1^D|
&\lesssim
B^{-1}
\left(
\|v\|_{\Sigma_{1AA_1}}
+
\|v\|_{\Sigma_{2AA_1}}
\right)
\left(
\|v\|_{\Sigma_{1AA_1}}
+
A_1^{-100}\|v\|_{\Sigma_{2AA_1}}
\right)
\\
&  \lesssim \varepsilon _1 \|v\|_{\Sigma_{1AA_1}}^2 +  A_1^{-1}\|v\|_{\Sigma_{2AA_1}}^2
  . \label{eq:estKD1}
\end{align}
Similarly,
\begin{align}  
|K_2^D|
&\lesssim
A_1^{-100}B^{-1}
\left(
\|v\|_{\Sigma_{1AA_1}}
+
\|v\|_{\Sigma_{2AA_1}}
\right)^2
  \lesssim
A_1^{-100}\|v\|_{\Sigma_{1AA_1}}^2
+
B^{-1}\|v\|_{\Sigma_{2AA_1}}^2.  \label{eq:estKD2}
\end{align}
Finally we focus on $K_i^c$. 
By the exponential decay of $\xi_2[c]$ and $\xi_2'[c]$,
\begin{align*}
\left|
\int_{\R}
\varphi_{1AA_1}
\left(
v\xi_2[c]+v'\xi_2'[c]
\right)\,dx
\right|
&\lesssim
\|v\|_{\widetilde{\Sigma}},
\\
\left|
\int_{\R}
\varphi_{2AA_1}
\left(
v\xi_2[c]+v'\xi_2'[c]
\right)\,dx
\right|
&\lesssim
A_1^{-100}
\left(
\|v\|_{\Sigma_{1AA_1}}
+
\|v\|_{\Sigma_{2AA_1}}
\right).
\end{align*}
Using   \eqref{eq:discrest1}, we obtain
\begin{align}
|K_1^c|
&\lesssim
B^{-1}
\left(
\|v\|_{\Sigma_{1AA_1}}
+
\|v\|_{\Sigma_{2AA_1}}
\right)
\|v\|_{\widetilde{\Sigma}}\lesssim
B^{-2}
 (
\|v\|_{\Sigma_{1AA_1}}^2
+
\|v\|_{\Sigma_{2AA_1}}^2 )
+
\|v\|_{\widetilde{\Sigma}}^2  
\label{eq:Kc1}
\end{align}
and 
\begin{align}
|K_2^c|
&\lesssim
A_1^{-100}B^{-1}
\left(
\|v\|_{\Sigma_{1AA_1}}
+
\|v\|_{\Sigma_{2AA_1}}
\right)^2
 .
\label{eq:Kc2}
\end{align}
Combining \eqref{eq:virial-decomposition},
\eqref{eq:K0-coercive2}, \eqref{eq:Kf2},  \eqref{eq:estKD2},  \eqref{eq:Kc2}  and absorbing $\|v\|_{\Sigma_{2AA_1}}^2$ on the left hand side,   we obtain
\begin{align}
\|v\|_{\Sigma_{2AA_1}}^2
\lesssim
-\frac{d}{dt}\mathcal I_2
+
A_1^{-2}
\|v\|_{\Sigma_{1AA_1}}^2
+
A_1^{-2}
\|v\|_{\widetilde{\Sigma}}^2 .
\label{eq:virial-second}
\end{align}
To prove \eqref{eq:lem:1stV11}, fix a sufficiently small constant
$\kappa>0$, depending only on $c_0$, and consider the sum of the two
identities in \eqref{eq:virial-decomposition} with coefficients
$1$ and $\kappa$. By \eqref{eq:K0-coercive2},
\begin{align*}
K_1^{(0)}+\kappa K_2^{(0)}
\leq
-\nu_0
\left(
\|v\|_{\Sigma_{1AA_1}}^2
+
\kappa\|v\|_{\Sigma_{2AA_1}}^2
\right)
+
CA_1^{-2}
\left(
\|v\|_{\Sigma_{1AA_1}}^2
+
\|v\|_{\Sigma_{2AA_1}}^2
\right).
\end{align*}
For $A_1$ sufficiently large, the last term is absorbed by the
coercive terms on the right. Using then \eqref{eq:Kf1},
\eqref{eq:Kf2}, \eqref{eq:estKD1}--\eqref{eq:estKD2},
and \eqref{eq:Kc1}--\eqref{eq:Kc2}, and taking $\epsilon$ sufficiently
small, we get
\begin{align*}
\|v\|_{\Sigma_{1AA_1}}^2
+
\kappa\|v\|_{\Sigma_{2AA_1}}^2
\lesssim
&
-\frac{d}{dt}\mathcal I_1
-
\kappa\frac{d}{dt}\mathcal I_2
+
\|v\|_{\widetilde{\Sigma}}^2 .
\end{align*}
This yields \eqref{eq:lem:1stV11}. From the last inequality and \eqref{eq:virial-second} we obtain   \eqref{eq:lem:1stV12}.
\end{proof}

\textit{Proof of Lemma \ref{prop:1virial}.}
We integrate
\eqref{eq:lem:1stV11}--\eqref{eq:lem:1stV12}
on the time interval $I$. By orbital stability, 
\begin{align}
\label{eq:virial-boundary}
|\mathcal I_i(t)|
&\leq
\|\varphi_{iAA_1}\|_{L^\infty}
\|v(t)\|_{H^1}^2
\lesssim
A  B \epsilon ^2   ,
\qquad i=1,2.
\end{align}
Hence
\begin{align*}
\|v\|_{L^2(I,\Sigma_{1AA_1})}
&\lesssim
\sqrt {AB}\,\epsilon
+
\|v\|_{L^2(I,\widetilde{\Sigma})}
 ,
\\
\|v\|_{L^2(I,\Sigma_{2AA_1})}
&\lesssim\sqrt {AB}\,\epsilon
+
A_1^{-1}
\|v\|_{L^2(I,\widetilde{\Sigma})}
 .
\end{align*}
By the hierarchy \eqref{eq:relABg}, after decreasing $\epsilon_0$ if
necessary,
$\sqrt {AB}\,\epsilon\leq \epsilon ^{1/3}$.
This proves
\eqref{eq:sec:1virial11}--\eqref{eq:sec:1virial12}.
\qed

\section{Dispersion and Kato smoothing  for the linearized equation }\label{sec:disp}

We will use weighted decay estimates for the linearized operator in \eqref{eq:Lin}, following Miller and Weinstein \cite{MillerWeinsteinBBM}.
  For $p>1$ and $c>c^*(p)$, we consider the operators
\begin{align}\label{eq:projections}
  \mathbf{P}_{[c]} := \sum _{i=1,2}   \xi _i [c]    \< \eta _i [c],  \cdot \>
 \text{ and }  \mathbf{Q}_{[c]}=1-  \mathbf{P}_{[c]} .
\end{align}  
We have the spectral decomposition  
\begin{align}
  \label{eq:specdec} L^2 _a = N_g\left ( \mathcal{L}_{[c]} \right ) \oplus N_g^\perp \left ( \mathcal{L}^*_{[c]} \right ) .
\end{align} 
We have the following, which follows for $p=2$  by \cite{MillerWeinsteinBBM} combined with \cite{MartelLinear}.  
 
  \begin{theorem}\label{thm:pwdisp}  For any  $ p>1$ and $c>c^*(p)$ there exists 
    $a_0>0$ and for   $a\in (0, a_0)$,
        there  exist  constants $C( a,p,c)>0$  and $ b(a,p,c)>0$ s.t.
   \begin{align}\label{eq:pwdisp1}
      \| e^{  t \mathcal{L}_{[c]} }  \mathbf{Q} _{[c]} v_0  \| _{L^2_a}\le C( a,p,c)    e^{-b(a,p,c) t}\|    \mathbf Q _{[c]}  v_0\| _{L^2_a}\text{ for all }v_0\in L^2_a.
   \end{align}

 \end{theorem}
 \begin{proof}
     The proof is the same as in \cite{MillerWeinsteinBBM}.
Indeed, the essential spectrum in $L_a^2(\R)$  does not depend on   $p$ and so it is the same as the case $p=2$. Moreover, Lemma \ref{lem:kernel} gives the same
two--dimensional generalized kernel at zero, while Proposition
\ref{prop:noembeig1} excludes all nonzero spectrum in
$\{\Re z\geq0\}$. Finally, the large  $|z|$ resolvent estimates of
\cite{MillerWeinsteinBBM} use only the exponential localization of
$f'(Q_c)$ and apply without change to
$f'(Q_c)=pQ_c^{p-1}$. The uniform resolvent bound on the range of
$\mathbf Q_{[c]}$ then follows, and the result follows from the
Pr\"uss theorem, \cite{Pruss1984},  used in \cite{MillerWeinsteinBBM}.
 \end{proof}

Next we have the following. 

\begin{proposition}\label{prop:smooth}  For any  $ p>1$,  $c>c^*(p)$ and $\sigma>1/2$  there is a   $C(  p,c,\sigma)>0$ such that
  \begin{align}\label{eq:smooth1} &
      \left \| \int _0 ^t    e^{  (t-s) \mathcal{L}_{[c]}  }  \mathbf{Q} _{[c]} f(s) ds  \right \| _{L^2 ( \R _+ ,    \widetilde{\Sigma}) }\le C(  p,c,\sigma)  \|     f \| _{ L^{ 2 }  ( \R  _+  ,    L^{ 2, \sigma + 1}( \R) ) }\text{ for all }f\in L^{ 2 }  ( \R  _+  ,    L^{ 2, \sigma + 1}( \R) )  .  
   \end{align}
\end{proposition}

Proposition \ref{prop:smooth} is a consequence of the following, more general result,  related to \cite{beceanuDuke}.   
\begin{proposition}\label{prop:smooth-0} For  any   $ p>1$, $c>c^*(p)$ and $\sigma>1/2$  there is a   $C(  p,c,\sigma)>0$ such that 
   \begin{align}\label{eq:smooth1-shifted} &
      \left \| \int _0 ^t    e^{  (t-s) \mathcal{L}_{[c]} -(t-s)\gamma ^{3/2}c\mathbf{P} _{[c]} }   f(s) ds  \right \| _{L^2 ( \R _+ ,    \widetilde{\Sigma}) }\le C(  p,c,\sigma)  \|     f \| _{ L^{ 2 }  ( \R  _+  ,    L^{ 2, \sigma + 1}( \R) ) }  .  
   \end{align}  
\end{proposition}

\begin{proof}
Here $\gamma^{3/2}c$ could be replaced by any positive
number. The proof is similar to the arguments in
\cite{CM25D1}.
By \cite{MillerWeinsteinBBM},
$e^{t\mathcal L_{[c]}}$, for $t\ge0$, is a $C_0$
semigroup in $L^2_a$. Since $\mathbf P_{[c]}$ is bounded
on $L^2_a$ and commutes with $\mathcal L_{[c]}$,
the operator
$\mathcal L_{[c]}-\gamma^{3/2}c\mathbf P_{[c]}$
also generates a $C_0$ semigroup in this space.
Using $\mathcal L_{[c]}^2\mathbf P_{[c]}=0$, we have
$$
\begin{aligned}
e^{t\mathcal L_{[c]}-t\gamma^{3/2}c\mathbf P_{[c]}}
={}&e^{t\mathcal L_{[c]}}\mathbf Q_{[c]}+e^{-t\gamma^{3/2}c}
\left(\mathbf P_{[c]}
+t\mathcal L_{[c]}\mathbf P_{[c]}\right) .
\end{aligned}
$$
Consequently, by Theorem \ref{thm:pwdisp}, this
semigroup decays exponentially in $L^2_a$.
Thus, see \cite[p. 240]{MR617913}, its resolvent
in $L^2_a$ is given by
\begin{equation}\label{eq:forresol}
\begin{aligned}
&R_{\mathcal L_{[c]}-\gamma^{3/2}c\mathbf P_{[c]}}(z)w_0=
-\int_0^{+\infty}
e^{-zt}
e^{t\mathcal L_{[c]}-t\gamma^{3/2}c\mathbf P_{[c]}}
w_0\,dt  \text{ for 
$\Re z>-\min\{b(a,p,c),\gamma^{3/2}c\}$.}
\end{aligned}
\end{equation}
Initially take
$f\in C_c^\infty((0,+\infty)\times\R)$.
Choose $a>0$ sufficiently small that
$L^2_a\subset\widetilde{\Sigma}$.
The exponential decay of the  semigroup
and the compact support of $f$ justify Fubini's theorem
in $L^2_a$,  giving
$$
\begin{aligned}
&\int_0^\infty e^{-\im\lambda t}
\int_0^t
e^{(t-s)\mathcal L_{[c]}
-(t-s)\gamma^{3/2}c\mathbf P_{[c]}}
f(s)\,ds\,dt\\
&\quad=
\int_0^\infty e^{-\im\lambda s}
\int_0^\infty e^{-\im\lambda r}
e^{r\mathcal L_{[c]}-r\gamma^{3/2}c\mathbf P_{[c]}}
f(s)\,dr\,ds =
-R_{\mathcal L_{[c]}-\gamma^{3/2}c\mathbf P_{[c]}}
(\im\lambda)\widehat f(\lambda).
\end{aligned}
$$
Applying Plancherel and Proposition \ref{prop:smooth2},
see below, gives
\begin{equation}\label{eq:fourtran}
\begin{aligned}
&\sqrt{2\pi}
\left\|
\int_0^t
e^{(t-s)\mathcal L_{[c]}
-(t-s)\gamma^{3/2}c\mathbf P_{[c]}}
f(s)\,ds
\right\|_{L^2(\R_+,\widetilde{\Sigma})} =
\left\|
R_{\mathcal L_{[c]}-\gamma^{3/2}c\mathbf P_{[c]}}
(\im\lambda)\widehat f(\lambda)
\right\|_{L^2(\R,\widetilde{\Sigma})} \\
&\le
C(p,c,\sigma)
\|\widehat f\|_{L^2(\R,L^{2,\sigma+1})} =
\sqrt{2\pi}\,C(p,c,\sigma)
\|f\|_{L^2(\R_+,L^{2,\sigma+1})}.
\end{aligned}
\end{equation}
By density, this extends to all 
$f\in L^2(\R_+,L^{2,\sigma+1})$ and proves
\eqref{eq:smooth1-shifted}.
\end{proof}

The following is proved in \S \ref{sec:proofsmooth2}.

\begin{proposition}\label{prop:smooth2} 
 For any  $ p>1$,  $c>c^*(p)$ and $\sigma>1/2$   there is a   $C(  p,c,\sigma)>0$ such that
  \begin{align}\label{eq:smooth21}&
     \sup  _{\lambda \in \R \backslash \{ 0  \}} \left \|    R_{\mathcal{L}_{[c]}-\gamma^{3/2}c\mathbf P_{[c]} }     ( \im \lambda   )       \right \| _{ \mathcal{L}\left (  L^{ 2, \sigma + 1}( \R) ,   \widetilde{\Sigma} \right )   }\le C( p,c,\sigma). 
   \end{align}
\end{proposition}

\section{Bounding $\| v \| _{\widetilde{\Sigma}}$  by Kato smoothing}\label{sec:smooth}

We prove Lemma \ref{prop:smooth11} assuming Proposition \ref{prop:smooth2}.
Multiplying equation \eqref{eq:gKdV3}  by $\zeta_{B}$
we obtain
\begin{align}     \nonumber
    \partial _t   \(  \zeta_{B}v    \)  &= \mathcal{L}_{[c_0]}    \(  \zeta_{B}v    \)  + [\zeta_{B}, \mathcal{L} _{[c_0]}]v
 + (\dot D -c )  \zeta_{B}  (   \xi _1[c]+   v' )- \dot c \  \zeta_{B}  \xi _2[c] \nonumber
\\&
+    \zeta_{B} \(   \mathcal{L}_{[c ]} -  \mathcal{L}_{[c_0]}   \) v     -   \zeta_{B} R_{-\partial _x ^2} (-1)\partial _x    \left ( f( Q   _{c} + v  )  -  f( Q  _{c} )-  f'(Q  _{c} ) v\right )     .   \label{eq:gKdV14}
\end{align}
Next we consider the spectral decomposition
 \begin{align}\label{eq:defw}
    \zeta_{B}v= \mathbf{P}_{[c_0]} \zeta_{B}v+ w \text { where }  w:=  \mathbf{Q}_{[c_0]} \zeta_{B}v .
 \end{align}
By essentially the same proof of \cite{CM25D1}, which we skip here,  we have the following.
\begin{lemma}
  \label{lem:disczbv} We have
\begin{align}
   \|   \mathbf{P}_{[c_0]} \zeta_{B}v \| _{L^2}\lesssim B ^{-1}  (  \| v  \|_{  { \Sigma }_{1A A_1}} + \| v  \|_{  { \Sigma }_{2A A_1}}  ) .\label{eq:disczbv}
\end{align}
\end{lemma}
\qed

We now focus on the   $w$   in \eqref{eq:defw}. Applying the projection $\mathbf Q_{[c_0]} $   to   \eqref{eq:gKdV14}
gives  
\begin{align}
   \dot w  &= \mathcal{L}_{[c_0]}   w +  \mathbf{Q}_{[c_0]} [\zeta_{B}, \mathcal{L}_{[c_0]} ]v \nonumber
\\& + (\dot D -c )   \mathbf{Q}_{[c_0]} \zeta_{B}  v'+
( \dot D -c )   \mathbf{Q}_{[c_0]} \zeta_{B}   \xi _1[c] - \dot c \   \mathbf{Q}_{[c_0]}  \zeta_{B} \xi _2[c]  \label{eq:gKdV22}
\\&
+   \mathbf{Q}_{[c_0]}
\zeta_{B} \(   \mathcal{L}_{[c ]} - \mathcal{L}_{[c_0]}   \) v   \label{eq:gKdV23}  \\&  -  \mathbf{Q} _{[c_0]}  \zeta_{B} R_{-\partial _x ^2} (-1)\partial _x    \left ( f( Q   _{c} + v  )  -  f( Q  _{c} )-  f'(Q  _{c} ) v\right )     .   \label{eq:gKdV24}
\end{align}
We next consider \small
\begin{align}\label{eq:expv1}
  w&=e^{  t \mathcal{L}_{[c_0]} }  w(0) +   \int _{0}^t e^{ (t-s) \mathcal{L} _{[c_0]}} \mathbf{Q} _{[c_0]}[\zeta_{B}, \mathcal{L}_{[c_0]} ]v ds  +     \int _{0}^t e^{  (t-s)\mathcal{L} _{[c_0]}} \text{lines \eqref{eq:gKdV22}--\eqref{eq:gKdV24}} ds.
\end{align}
\normalsize
The following can be proved almost exactly like in \cite{CM25D1}, so we skip the proof.

\begin{lemma}\label{lem:w0} We have
   \begin{align*}
     & \| e^{  t \mathcal{L}_{[c_0]} }  w(0)\| _{L^2 (\R _+, \widetilde{\Sigma})}\le  \sqrt{\epsilon} , \\& 
     \left  \|  \int _{0}^t e^{  (t-s)\mathcal{L}_{[c_0]} } \text{line  \eqref{eq:gKdV22}} ds\right \| _{L^2 (I, \widetilde{\Sigma})}\lesssim   \|  \dot D -c \| _{L^2 \( I  \) }+ \|  \dot  c \| _{L^2 \( I  \) }   , \\& 
     \left  \|  \int _{0}^t e^{  (t-s)\mathcal{L}_{[c_0]} } \text{line  \eqref{eq:gKdV23}} ds\right \| _{L^2 (I, \widetilde{\Sigma})}\lesssim  \sqrt{ \epsilon}  (\| v \| _{L^2(I,  { \Sigma }_{1A A_1}   )}+\| v \| _{L^2(I,  { \Sigma }_{2A A_1}   )}),\\&   \left  \|  \int _{0}^t e^{  (t-s)\mathcal{L}_{[c_0]} } \text{line  \eqref{eq:gKdV24}} ds\right \| _{L^2 (I, \widetilde{\Sigma})}\lesssim   \epsilon ^{  \min\{ \frac{p-1}{2} , 1/2\}   } (\| v \| _{L^2(I,  { \Sigma }_{1A A_1}   )}+\| v \| _{L^2(I,  { \Sigma }_{2A A_1}   )}) .
   \end{align*}
\end{lemma}
\qed

\begin{lemma}\label{lem:gKdV21hard} For any fixed $\sigma >1/2$
   \begin{align}\label{eq:gKdV21hard1}
     \left  \|  \int _{0}^t e^{  (t-s)\mathcal{L}_{[c_0]}}  \mathbf{Q} _{[c_0]}  [\zeta_{B}, \mathcal{L} _{[c_0]}] v     ds\right \| _{L^2 (I, \widetilde{\Sigma})}\lesssim   A_1 ^{\sigma + 1} B^{-1}    \| v \| _{L^2(I,  { \Sigma }_{1A A_1}   )}+ B ^{\sigma  } \| v \| _{L^2(I,  { \Sigma }_{2A A_1}   )} .
   \end{align}
 \end{lemma}
\begin{proof}
By elementary computations we have  \small
\begin{align*}
  &[\zeta_{B}, \mathcal{L} _{[c_0]}] v = c_0 R_{-\partial _x ^2} (-1) \partial _x (\zeta _B'' +2\zeta _B' \partial _x)v
  -\zeta _B'R_{-\partial _x ^2} (-1)  L_{c_0}v \\& 
  - R_{-\partial _x ^2} (-1) (\zeta _B''' +2\zeta _B'' \partial _x)R_{-\partial _x ^2} (-1)   L_{c_0}v
  - R_{-\partial _x ^2} (-1) (\zeta _B'' +2\zeta _B' \partial _x) R_{-\partial _x ^2} (-1)\partial _xL_{c_0}v.
\end{align*}
\normalsize
We will bound only the contribution of the term $\zeta _B'R_{-\partial _x ^2} (-1)  L_{c_0}v$  since other ones are either simpler, when they involve $\zeta _B^{(k)}$ for $k\ge2$, or can be reduced to similar computations.  We have 
\begin{align} \nonumber
  &\left  \|  \int _{0}^t e^{  (t-s)\mathcal{L}_{[c_0]}}  \mathbf{Q} _{[c_0]}  \zeta _B'R_{-\partial _x ^2} (-1)  L_{c_0}v     ds\right \| _{L^2 (I, \widetilde{\Sigma})}\lesssim \|     \zeta _B'R_{-\partial _x ^2} (-1)  L_{c_0}v \| _{ L^{ 2 }  ( I  ,    L^{ 2, \sigma + 1}( \R) ) }\\& \le
  c_0\|     \zeta _B' v \| _{ L^{ 2 }  ( I  ,    L^{ 2, \sigma + 1}( \R) ) }+
  c_0\|     \zeta _B'     e ^{-|\cdot |}*v \| _{ L^{ 2 }  ( I  ,    L^{ 2, \sigma + 1}( \R) ) }\label{eq:gKdV21hard11}\\& +
    c_0\|     \zeta _B'     e ^{-|\cdot |}* (f'(Q_{c_0})   v )\| _{ L^{ 2 }  ( I  ,    L^{ 2, \sigma + 1}( \R) ) }.\label{eq:gKdV21hard12}
\end{align}
By $|\zeta _B' |\lesssim B^{-1}\zeta _B$  and using $1= \vartheta _{1A_1} + \vartheta _{2A_1} $, we have
\begin{align*}
  & 
   \|  \langle  x\rangle ^{ \sigma + 1}   \zeta _B' (\vartheta _{1A_1} + \vartheta _{2A_1}) v \| _{ L^{ 2 }  ( I  ,    L^{ 2}( \R) ) }  \lesssim  
    A_1 ^{\sigma + 1} B^{-1}    \| v \| _{L^2(I,  { \Sigma }_{1A A_1}   )}+ B ^{\sigma  } \| v \| _{L^2(I,  { \Sigma }_{2A A_1}   )}.
\end{align*}
The following implies that the second term in \eqref {eq:gKdV21hard11} is similar, 
\begin{align*}
  & 
   \|  \langle  x\rangle ^{ \sigma + 1}   \zeta _B'  e ^{-|\cdot |}* v \| _{     L^{ 2}( \R)  }  \lesssim  B^{-1} \|  \langle  x\rangle ^{ \sigma + 1}   \zeta _B     v \| _{     L^{ 2}( \R)  } .
\end{align*}
The term in \eqref {eq:gKdV21hard12} has smaller upper bound, thanks to the exponential decay of $f'(Q_{c_0})$.

\end{proof}

The following is essentially the same as the corresponding argument in \cite{CM25D1}.

\textit{Proof of Lemma \ref{prop:smooth11}.} From   Lemmas \ref{lem:lemdscrt}, \ref{lem:w0} and  \ref{lem:gKdV21hard}
we have, for $a_+:=\min  \{ \frac{p-1}{2} , 1/4  \}$
\begin{align*}
  \| w \| _{L^2(I, \widetilde{\Sigma}   )} \lesssim A_1 ^{\sigma + 1} B^{-1}    \| v \| _{L^2(I,  { \Sigma }_{1A A_1}   )}+ B ^{\sigma  } \| v \| _{L^2(I,  { \Sigma }_{2A A_1}   )} + \epsilon ^{a_+} .
\end{align*}
 Using  Lemma \ref{prop:1virial}, \eqref{eq:relABg}, decomposition \eqref{eq:defw} and inequality \eqref{eq:disczbv}, we have the following
\begin{align*}
  & \| v \| _{L^2(I, \widetilde{\Sigma}   )}  \le      \| (1-\zeta_{B}) v \| _{L^2(I, \widetilde{\Sigma}   )}  + \| \zeta_{B} v \| _{L^2(I, \widetilde{\Sigma}   )}   \\& \lesssim B ^{-1} \(  \| v \| _{L^2(I,  { \Sigma }_{1A A_1}   )} + \| v \| _{L^2(I,  { \Sigma }_{2A A_1}   )}    \) + \| \mathbf P_{[c_0]}\zeta_Bv  \| _{L^2(I, \widetilde{\Sigma}   )} +\|   w \| _{L^2(I, \widetilde{\Sigma}   )}  \\& \lesssim   A_1 ^{\sigma + 1} B^{-1}    \| v \| _{L^2(I,  { \Sigma }_{1A A_1}   )}+ B ^{\sigma  } \| v \| _{L^2(I,  { \Sigma }_{2A A_1}   )} + \epsilon ^{a_+}\\& \lesssim  
   \left (   A_1 ^{\sigma + 1} B^{-1} + B ^{\sigma  }A_1^{-1} \right ) \|v\|_{L^2(I,\widetilde{\Sigma})}
+ \epsilon ^{a_+},
\end{align*}
which, for  $ A_1=B^{3/5}$  and $1/2 < \sigma < 3/5$, yields \eqref{eq:sec:smooth11}.
\qed

\section{Resolvent estimates: proof of {Proposition} \ref{prop:smooth2}   }  \label{sec:proofsmooth2}

We consider the scaling change
\begin{align}
    \label{eq:deltac}  \delta _c v:=\widetilde{v}   \text{ for $\widetilde{v}$ defined by }    v(y)=\widetilde{v} (\gamma ^{1/2} y). \end{align}
Besides the operator $\mathcal{L}_{[c]}$  we will consider
\begin{align}\label{eq:tildLin}
   &\widetilde{\mathcal{L}} _{[c]}: =(1-\gamma \partial _x ^2) ^{-1} \partial_xL_+
  \text{ where }
 L_+:=-\partial_x^2+1-f'(Q )  , \\&  \label{eq:ALin}  A_{[c,\lambda ]}:=\partial_xL_c-\lambda (1- \partial _x ^2)  \text{ and }
 \widetilde{A}   _{[c,\lambda ]}:=\partial_xL_+-\lambda (1- \gamma\partial _x ^2).
\end{align}
Furthermore  we split 
\begin{align}\label{eq:ALinsplit}
   & 
 \widetilde{A}   _{[c,\lambda ]} =   \widetilde{A}   _{[c,\lambda ]} ^{(0)}   +T     \text{ with } \widetilde{A}   _{[c,\lambda ]} ^{(0)}  
 :=- \partial _x ^3+\partial _x -\lambda (1- \gamma\partial _x ^2) \text{ and } T =-\partial_x\circ f'(Q).
\end{align}
It is elementary that  $ \delta _c \mathcal{L}_{[c]} =\gamma  ^{3/2}  c  \widetilde{\mathcal{L}} _{[c]}  \delta _c $ and 
$ \delta _c    A_{[c,\lambda ]} =\gamma  ^{3/2}  c   \widetilde A_{[c,z ]}  \delta _c $ for $\lambda = \gamma  ^{3/2}  c \ z $.
Notice that     $\delta _c : L^{2,s}(\R )\to L^{2,s}(\R )$,  $\delta _c:L^2_{a}(\R ) \to L^2_{\widehat a}(\R )$   for $a= \widehat a   \gamma  ^{1/2}  $  and $\delta _c:\widetilde{\Sigma}\to \widehat{\Sigma}$ are all isomorphisms. Furthermore  $\delta _c:L^2_{a}(\R ) \to L^2_{\widehat a}(\R )$  preserves the decompositions 
\begin{align}
     L^2 _a (\R )= N_g\left ( \mathcal{L}_{[c]} \right ) \oplus N_g^\perp \left ( \mathcal{L}^*_{[c]} \right ) \xrightarrow{\delta _c}  L^2 _{\widehat a} (\R )= N_g\left ( \widetilde{\mathcal{L}} _{[c]} \right ) \oplus N_g^\perp \left ( \widetilde{\mathcal{L}} _{[c]} ^*  \right ).
\end{align}
The projections associated to the direct sum on the right are $\widetilde {\mathbf{P}}_{[c]} = \delta _c \mathbf{P}_{[c]}\delta _c ^{-1} $
and   $\widetilde {\mathbf{Q}}_{[c]} = \delta _c \mathbf{Q}_{[c]}\delta _c ^{-1} $ where  \begin{align}\label{eq:tildproj}
 \widetilde {\mathbf{P}}_{[c]} := \sum _{i=1,2}  \delta _c \xi _i [c]    \<  (\delta _c ^{-1} )^*\eta _i [c],  \cdot \>
 \text{ and }  \widetilde {\mathbf{Q}}_{[c]}=1-   \widetilde {\mathbf{P}}_{[c]} .
\end{align}
It is easy to see that {Proposition} \ref{prop:smooth2} can be equivalently reframed as follows. 
\begin{proposition}\label{prop:smooth21} 
 For  any  $ p>1$, $c>c^*(p)$ and $\sigma>1/2$  there is a   $C(  p,c,\sigma)>0$ such that
  \begin{align}\label{eq:smooth211}&
     \sup  _{\lambda \in \R \backslash \{ 0  \}} \left \|    R_{\widetilde{\mathcal{L}}_{[c]} - \widetilde {\mathbf{P}}_{[c]}}     ( \im \lambda   )     \right \| _{ \mathcal{L}\left (  L^{ 2, \sigma + 1}( \R) ,   \widehat{\Sigma} \right )   }\le C( p,c,\sigma).  
   \end{align}
\end{proposition}
In order to prove {Proposition} \ref{prop:smooth21} we consider  the equation 
\begin{align}
     \widetilde A_{[c,z]} u=0 \label{eq:lineq} ,
\end{align}
whose characteristic equation   is
\begin{align}\label{eq:mus}
  P_\gamma(\mu,z) :=  \mu^3-\gamma z\mu^2-\mu+z=0.
\end{align}
We summarize some properties of the roots of \eqref{eq:mus}.

\begin{lemma}\label{lem:roots1}
\begin{description}
  \item[a]   For 
  \begin{align}\label{eq:roots11}
      z= \im  \frac{k(1+k^2)}{1+\gamma k^2}   \text{  where } k\in \R,
  \end{align}
  \eqref{eq:mus} has the three roots 
\begin{align}
    \mu_2(z,\gamma)
    &=
    \im k,
    \label{eq:mu2_imag}\\
    \mu_1(z,\gamma)
    &=
    -\alpha_\gamma(k)
    -
    \im\frac{(1-\gamma)k}
    {2(1+\gamma k^2)},
    \label{eq:mu1_imag}\\
    \mu_3(z,\gamma)
    &=
    \alpha_\gamma(k)
    -
    \im\frac{(1-\gamma)k}
    {2(1+\gamma k^2)},
    \label{eq:mu3_imag}
\end{align}
where
\begin{align}\label{eq:alpha_gamma}
    \alpha_\gamma(k)
    :=
    \left[
        \frac{1+k^2}{1+\gamma k^2}
        -
        \frac{(1-\gamma)^2k^2}
        {4(1+\gamma k^2)^2}
    \right]^{1/2}
    .
\end{align}

\item[b] For any  $\gamma\in(0,1)$ there exists $\varepsilon_\gamma>0$ such that,
for 
\begin{align}
    \label{eq:Sgamma}  S_\gamma
    :=
    \left\{
        z\in\C:
        |\Re z|<\varepsilon_\gamma
    \right\},
\end{align}
there  are three   functions  $\mu _j \in C^\omega ( S_\gamma , \C )$  giving for any $z\in S_\gamma$ the three roots of   \eqref{eq:mus},    which extend 
the roots in Claim \textbf{a}.

\item[c] For  $ 0<\Re z<\varepsilon_\gamma,$  we have
\begin{align}\label{eq:root_order}
    \Re\mu_1(z,\gamma)<0<\Re\mu_2(z,\gamma)<\Re\mu_3(z,\gamma).
\end{align}
Moreover, after possibly decreasing $\varepsilon_\gamma$, there exists
$\kappa_\gamma>0$ such that   for all $z\in\overline{S_\gamma}$
\begin{align}
    \Re\mu_1(z,\gamma)
    &\leq-\kappa_\gamma,
    \label{eq:root_gap1}\\
    \Re\bigl(\mu_2(z,\gamma)-\mu_1(z,\gamma)\bigr)
    &\geq\kappa_\gamma,
    \label{eq:root_gap2}\\
    \Re\bigl(\mu_3(z,\gamma)-\mu_2(z,\gamma)\bigr)
    &\geq\kappa_\gamma .
    \label{eq:root_gap3}
\end{align}

\item[d] For $z$ near 0   
\begin{align}
    \mu_1(z,\gamma)
    &=
    -1-\frac{1-\gamma}{2}z+O(z^2),
    \label{eq:mu1_low}\\
    \mu_2(z,\gamma)
    &=
    z+O(z^3),
    \label{eq:mu2_low}\\
    \mu_3(z,\gamma)
    &=
    1-\frac{1-\gamma}{2}z+O(z^2).
    \label{eq:mu3_low}
\end{align}
\item[e]    For  $z\to\infty$ in $  \overline{S}_\gamma$  we have
\begin{align}
    \mu_1(z,\gamma)
    &=
    -\gamma^{-1/2}+O (|z|^{-1}),
    \label{eq:mu1_high}\\
    \mu_2(z,\gamma)
    &=
    \gamma z+O (|z|^{-1}),
    \label{eq:mu2_high}\\
    \mu_3(z,\gamma)
    &=
    \gamma^{-1/2}+O (|z|^{-1}).
    \label{eq:mu3_high}
\end{align}
\item[f] The following identities hold, 
\begin{align}
    \mu_1+\mu_2+\mu_3
    &=
    \gamma z,
    \label{eq:sum_of_mus}\\
    \mu_1\mu_2+\mu_1\mu_3+\mu_2\mu_3
    &=
    -1,
    \label{eq:sum_of_products}\\
    \mu_1\mu_2\mu_3
    &=
    -z.
    \label{eq:product_of_mus}
\end{align}
\item[g]   We have \begin{align}
    \label{eq:mu1mu3}   \mu _1(-z, \gamma )=- \mu _3(z, \gamma ) .
 \end{align}

\end{description}

\end{lemma}

\begin{proof}
Substituting  $\mu =\mu_2:=
    \im k$  in \eqref{eq:mus} and solving \eqref{eq:mus}   in $z$,     we obtain   \eqref{eq:roots11}. Then
\begin{align*}
    \mu^3-\gamma z\mu^2-\mu+z= (\mu -\im k) \left (  \mu ^2 +\im k \frac{1-\gamma}{1+\gamma k^2}   \mu -  \frac{ 1+k^2}{1+\gamma k^2}    \right ) .
\end{align*}
The roots of the quadratic  polynomial yield the two additional  roots  in \eqref{eq:mu1_imag}--\eqref{eq:mu3_imag}. This proves Claim \textbf{a}.
    
To prove  Claim \textbf{b}, we know already from the part  \textbf{a} that the three roots, explicitly written in \eqref{eq:mu2_imag}--\eqref{eq:mu3_imag}, are distinct. We need to show that this remains true for $z\in S_\gamma$, if we pick 
$\varepsilon_\gamma $ small enough.  
Equation \eqref{eq:mus} is equivalently $z=\frac{\mu   -\mu ^3}{1-\gamma \mu ^2}$ while $\partial _\mu  P_\gamma(\mu,z )=0$ can be written as 
$z=\frac{3\mu ^2 -1}{2\gamma \mu  }$, so that equating we obtain $    \gamma \mu ^4  -( 3-\gamma )  \mu ^2 +1=0  $.
The latter equation has at most four solutions $\mu$, which entered in    $z=\frac{\mu   -\mu ^3}{1-\gamma \mu ^2}$ give at most four distinct values of $z$, all   outside $\im \R$  and, picking $\varepsilon_\gamma $ small enough,  also outside $  S_\gamma$.    Hence, since in the simply connected  open set $S_\gamma$ all the roots of $ P_\gamma(\mu ,z )$ are simple, we obtain  
Claim \textbf{b}  by Corollary of Theorem XII.2 in \cite{ReedSimon4}.
 
By \eqref{eq:mu2_imag}--\eqref{eq:mu3_imag}, $\mu _1(0)=-\mu _3(0) =-1$  and $\mu _2(0)=0$. Then \eqref{eq:mu1_low}--\eqref{eq:mu3_low} follow by elementary argument.  
Since by Claim \textbf{a} an imaginary  root $\mu $ can occur only for   imaginary
$z$,   for  $\Re z>0$, the number of roots in the
 left and right half-planes is constant. This and  \eqref{eq:mu1_low}--\eqref{eq:mu3_low} for small $z>0$ yield \eqref{eq:root_order}  for $ 0<\Re z<\varepsilon_\gamma$.

Claim \textbf{f} is elementary.    Claim \textbf{g}   follows from   $ P_\gamma(-\mu,-z)=-P_\gamma(\mu,z)$. 
Turning to Claim \textbf{e}, notice by  \eqref{eq:mu3_imag} and \eqref{eq:alpha_gamma} that formula \eqref{eq:mu3_high} is true for $z\in \im \R$. Introducing the new variable $\zeta =z ^{-1}$, by the implicit function theorem there is a solution of
\begin{align*}
    z   \left [ 1-\gamma   \mu ^2 +\zeta (\mu ^3 -\mu )         \right  ] =0 \text{ of the form } \mu = \gamma ^{-1/2} +O(\zeta )
\end{align*}
which by \eqref{eq:root_order} is $\mu _3$, yielding \eqref{eq:mu3_high}.  Then \eqref{eq:mu1_high} follows from  \eqref{eq:mu3_high} by \eqref{eq:mu1mu3} and   \eqref{eq:mu2_high} from \eqref{eq:sum_of_mus}. Finally, \eqref{eq:root_gap1}--\eqref{eq:root_gap3} are obtained  from behaviors near infinity, near zero, and by \eqref{eq:root_order}.

\end{proof}

By the variation of parameter formulas,   the solutions of the nonhomogeneous equation
\begin{align}\label{eq:nonhom}
   \widetilde A_{[c,z ]} ^{(0)} y:=-y''' +\gamma zy''+y' -z   y=F
\end{align}
are \begin{align}\label{eq:solution_nonhomogeneous_equation}
   & y =
    \sum_{ {j}=1,2,3 } a_ {j}  e^{\mu_ {j}  x}\int \ e^{-\mu_ {j}  x } F(x)   d x  ,
     \ \text{ where} \\& \nonumber    a_1 = \frac{ -1}{(\mu _1-\mu _2)(\mu _1-\mu _3) }  \text{ , }
         a_2   =  \frac{-1 }{(\mu _2-\mu _1)(\mu _2-\mu _3)} \text{ and  }
          a_3 = \frac{-1 }{(\mu _1-\mu _3)(\mu _2-\mu _3)} \ .
\end{align}
For $F\in L^2(\R )$  and $0< \Re z< \varepsilon _\gamma $  we have $L^2(\R )\ni y=R_{0} (z) F= R_{0} (z , \cdot )*F $   with the convolution with
\begin{align}\label{eq:freeres}
  R_{0} (z , x ): =  a_1(z, \gamma )   e^{\mu _1(z, \gamma )x } 1_{\R _+}(x ) - 1_{\R _-}(x ) \sum _{j=2,3}a_j(z, \gamma  )   e^{\mu _j(z, \gamma )x } 
\end{align}
where by $\sum _{1\le j \le 3}a_j=0= \sum _{1\le j \le 3}a_j\mu _j$  we have for the distributional derivatives \small
\begin{align}\label{eq:freeresder}
 \partial _x ^\alpha  R_{0} (z ,x) =  a_1(z, \gamma ) \mu _1^\alpha(z, \gamma )  e^{\mu _1(z, \gamma )x } 1_{\R _+}(x ) - 1_{\R _-}(x ) \sum _{j=2,3}a_j(z, \gamma  ) \mu _j^\alpha(z, \gamma )  e^{\mu _j(z, \gamma )x }  \text{ for }\alpha =1,2.
\end{align}
\normalsize
Then $ R_{0} (z )  \partial _x ^\alpha     =  \partial _x ^\alpha   R_{0} (z )  $  for $\alpha =1,2$. 
 For $s>1/2$ we introduce the space 
\begin{align*}
    Y_s:=  \{ v\in    L^{2,-s}(\R )   : v|_{\R _+} \in L^1(\R _+)\}  \text{ with }\| v\| _{ Y_s}:= \| v\| _{  L^{2,-s}(\R )} + \| v\| _{  L^1(\R _+)}.
\end{align*}
Obviously $ Y_s\subset \widehat{\Sigma}$.  Let 
$S_\gamma ^+$ be  the strip in $\C$ defined by $0\le \Re z < \varepsilon _\gamma$.

\begin{lemma}
\label{lem:R0-L1-L2ms} 
For any  $s,\sigma >1/2$  there is a  constant $C (s,\sigma , \gamma )>0$  such that for  $z  \in S_\gamma ^+$ 
\begin{align}
& \| \partial ^\alpha _xR_0 (z) \|_{\mathcal{L}( L^{2,\sigma +1}(\R ), Y_s)}\le C (s,\sigma , \gamma )
\langle z\rangle ^{-1}  \text{ for $\alpha =0,1$ and } \label{eq:R0-L1-tildesigma}
\\  &\| \partial ^2 _x  R_0 (z) \|_{\mathcal{L}( L^{2,\sigma +1}(\R ), Y_s)}\le C (s,\sigma , \gamma )  .\label{eq:R0-L1-W1inf-asymp}
\end{align}
We have $(z\to \partial _x^\alpha R_0 (z))\in C^0(  S_\gamma ^+, \mathcal{L}( L^{2,\sigma +1}(\R ), Y_s)) $  for $\alpha \le 2 $. 

\noindent  Moreover, for every $0<\widehat a<\kappa_\gamma$ and
$z\in S_\gamma^+$, there exists $C(z,\widehat a)>0$ such that
\begin{align}\label{eq:R0-weighted-L2}
\sum_{\alpha=0}^2
\|\partial_x^\alpha R_0(z)u\|_{L^2_{\widehat a}}
\le C(z,\widehat a)\|u\|_{L^2_{\widehat a}}
\qquad\text{for all }u\in L^2_{\widehat a}(\mathbb R).
\end{align}

\end{lemma}
\begin{proof}  For a   $C(z)>0$, by \eqref{eq:freeres}   for the kernel  of $\partial ^\alpha _xR_0 (z )$ with $\alpha \le 2$  we have 
\begin{align}\label{eq:estres0}
    |\partial ^\alpha _xR_0 (z,x )|\le C(z) \left ( 1_{\R _+} (x) e^{-\kappa _\gamma |x|}  + 1_{\R _-} (x)   \right ).
\end{align}
It is clear that the convolution by $1_{\R _+} (x) e^{-\kappa _\gamma |x|}$ is in $\mathcal{L}( L^{2,\sigma +1}(\R ), Y_s)$.  We have
\begin{align*}
     \| 1_{\R _-}*  v \| _{L^{2,-s}(\R )}\lesssim \| 1_{\R _-} * v \| _{L^{\infty }(\R )} \le  \|   v \| _{L^{1 }(\R )} \lesssim  \|   v \| _{L^{2,\sigma }(\R )}.
\end{align*}
We have the following, which shows the need of $v\in L^{2,\sigma +1}(\R )$, 
\begin{align*}
     \| 1_{\R _-}*  v \| _{L^{1}(\R _+)}\le\int _0^{+\infty}dx \int _x^{+\infty} |v(y)| dy =  \int _0^{+\infty} y |v(y)| dy\lesssim  \|   v \| _{L^{2,\sigma +1}(\R )}.
\end{align*}
Then the bounds in \eqref{eq:R0-L1-tildesigma}--\eqref{eq:R0-L1-W1inf-asymp} follow from the bounds in   Lemma \ref{lem:roots1}.  In the above estimates 
there is some room in the $\sigma $ which allows to exploit  bounds of the following  form, for small $\beta >0$, 
\begin{align*}
    | e^{\mu _j(z, \gamma )(x-y) } - e^{\mu _j(\zeta , \gamma )(x-y) }|\lesssim |\mu _j(z, \gamma ) - \mu _j(\zeta, \gamma )   |^\beta   |  x-y|^\beta .
\end{align*}
 The inequality provides in fact, repeating previous arguments,  local $C^{0, \beta}$ continuity of the maps  $z\to \partial _x^\alpha R_0 (z)  $. 

Finally, we prove \eqref{eq:R0-weighted-L2}.
By \eqref{eq:estres0}, for $0\le\alpha\le2$ and
$0<\widehat a<\kappa_\gamma$,
$$
\|e^{\widehat a x}\partial_x^\alpha R_0(z,x)\|_{L^1_x}
\le  C(z) \left ( \frac{1}{\kappa_\gamma-\widehat a}
+\frac{1}{\widehat a}
\right) .
$$
For $u\in L^2_{\widehat a}(\mathbb R)$, we have
$ 
e^{\widehat a x}\partial_x^\alpha R_0(z)u
=
\bigl(e^{\widehat a\cdot}\partial_x^\alpha R_0(z,\cdot)\bigr)
*
\bigl(e^{\widehat a\cdot}u\bigr).
$ 
Young's inequality   yields the following, proving \eqref{eq:R0-weighted-L2},
$$
\|\partial_x^\alpha R_0(z)u\|_{L^2_{\widehat a}}
\le  C(z) \left ( \frac{1}{\kappa_\gamma-\widehat a}
+\frac{1}{\widehat a}
\right)  \| u\|_{L^2_{\widehat a}}.
$$ 

\end{proof}

\begin{proof}[Proof of Proposition \ref{prop:smooth21}]
    
By \eqref{eq:tildLin}--\eqref{eq:ALinsplit} we   write 
\begin{align}\label {eq:smooth23large}
       \widetilde{\mathcal{L}}_{[c]} - \widetilde {\mathbf{P}}_{[c]} -z= (1-\gamma \partial ^2_x) ^{-1}  \widetilde{A}   _{[c,z ]} ^{(0)} ( 1+ K(z) ) \text{ , } K(z):=R_0(z) T- R_0(z)   (1-\gamma \partial ^2_x) \widetilde {\mathbf{P}}_{[c]}.
\end{align}
 By Lemma \ref{lem:R0-L1-L2ms} we have $(z\to K(z)) \in C^0(  S_\gamma ^+, \mathcal{L}(   Y_s))$.
 We claim that $K(z)$ is compact on $Y_s$ for every
$z\in S_\gamma^+$.  Since $\text{rank} \ \widetilde {\mathbf{P}}_{[c]}=2$, it  is enough to focus on $R_0(z) T$.  
  Since $f'(Q)$ and all
its derivatives decay exponentially,  as a multiplication operator 
$f'(Q) \in \mathcal{L}( Y_s,  L^{2,\sigma+1}(\R)) $ for every
$\sigma>1/2$. Hence, by Lemma \ref{lem:R0-L1-L2ms},
 $ \partial_x^\alpha R_0(z) \circ f'(Q) \in \mathcal{L}( Y_s )$  for $\alpha =1,2$.  By the Rellich theorem,
for   any $C>0$ we conclude that $ \chi _C   R_0(z) T:Y _s\to Y _s$ is compact, with $\chi _C$ like in \eqref{eq:chi}.
Using the bound \eqref{eq:estres0} it is easy to see that $\| (1-\chi _C)   R_0(z)T\| _{\mathcal{L}(Y_s)} \xrightarrow{C\to +\infty}0 $. Indeed   
\begin{align*}   \| \< x \> ^{-s} (1-\chi _C)   R_0(z)Tu \| _{L^2}& \lesssim 
    \| \< x \> ^{-s} (1-\chi _C)  ( 1_{\R _+}   e^{-\kappa _\gamma | \cdot|}  + 1_{\R _-}  ) * |f'(Q)u| \| _{L^2}   \\& \le   \| \< x \> ^{-s} (1-\chi _C)     \| _{L^2} \|  f'(Q)u    \| _{L^1}
  \lesssim  \| \< x \> ^{-s} (1-\chi _C)     \| _{L^2} \|   u    \| _{ L^{2,-s}}
\end{align*}
where by $s>1/2$ we have $ \| \< x \> ^{-s} (1-\chi _C)    \| _{L^2}\xrightarrow{C\to +\infty}0$,
so that $\| (1-\chi _C)   R_0(z)T\| _{\mathcal{L}(L^{2,-s})} \xrightarrow{C\to +\infty}0$.
Set $\rho=\frac12\min\{\kappa_\gamma,(p-1)/2\}$. By  \eqref{eq:estres0}, 
\begin{align*}&
     \|(1-\chi _C) 1_{\R _-} * f'(Q)u  \| _{L^{1}(\R _+)}  \le \int _C^{+\infty}dx \int _x^{+\infty} |  f'(Q (y))    u (y)| dy \\& \le  \int _C^{+\infty} y      |f'(Q (y))  u(y)| dy\lesssim  C e^{-C\rho}     \|   u \| _{L^{ 1}(\R _+ )}.
\end{align*}
By $e^{-\kappa _\gamma |x -y|} f'(Q (y))\lesssim e^{-\rho (|x| +|  y|)} $ we obtain
\begin{align*}&
     \|(1-\chi _C) 1_{\R _+}   e^{-\kappa _\gamma | \cdot |}  * f'(Q)u  \| _{L^{1}(\R _+)}  \le \int _C^{+\infty}dx \int _ {-\infty} ^x  e^{-\kappa _\gamma |x -y|}   |  f'(Q (y))    u (y)| dy \\&  \lesssim 
     \int _C^{+\infty}dx e^{-\rho  x } \int _ {\R }   e^{-\rho  |  y|}   |      u (y)| dy
      \lesssim   e^{-C\rho }     \|   u \| _{Y_s }.
\end{align*}
Hence we conclude  also $\| (1-\chi _C)   R_0(z)T\| _{\mathcal{L}(Y_s, L^{1}(\R _+))} \xrightarrow{C\to +\infty}0 $, and we   conclude that 
$R_0(z)T:Y_s\to Y_s$ is compact.
  Suppose now $( 1+ K(z) )\widetilde g =0$ for a nonzero $\widetilde g\in  Y_s$. Then,    
 since for $\widehat a>0$ small  \small
\begin{align*}
    e^{\widehat a x}|\partial _x^\alpha R_0(z,x-y) f'(Q(y) ) | \langle y \rangle ^s \lesssim   e^{\widehat a x} \left ( 1_{\R _+} (x-y) e^{-\kappa _\gamma |x-y|}  + 1_{\R _-} (x-y)   \right ) e^{-2\rho|y|} \langle y \rangle ^s \in L^2(\R \times \R )
\end{align*}
\normalsize
for $\alpha \le 2$, it follows that $  R_0(z)T \in \mathcal{L} (L^{2,-s} (\R ) , H^1_{\widehat a} (\R )) $. It is easy to conclude that  $K(z) \in \mathcal{L} (Y_s ,H^1_{\widehat a} (\R ))$.
Then   $\widetilde  g\in H^1_{\widehat a} (\R )$.  We have 
\begin{align}\label{eq:notildeg}
0=    R_0(z) (   \widetilde{A}   _{[c,z ]} -  (1-\gamma \partial ^2_x) \widetilde {\mathbf{P}}_{[c]} ) \widetilde g= (1-\gamma \partial ^2_x) R_0(z) ( \widetilde{\mathcal{L}} _{[c]} -z-     \widetilde {\mathbf{P}}_{[c]} )\widetilde g  ,
\end{align}
where $( \widetilde{\mathcal{L}} _{[c]} -z-     \widetilde {\mathbf{P}}_{[c]} )\widetilde g \in L^2_{\widehat a} (\R )$ by  
$\widetilde  g\in H^1_{\widehat a} (\R )$ and where the right hand side is well defined by 
\eqref{eq:R0-weighted-L2}.

We have $\ker ( 1-\gamma \partial ^2_x) =0$   because any nonzero linear combination of $e^{\pm \frac{x}{\sqrt{\gamma}}  }$ 
is, by $\widehat a<1$,  not 
in   $  L^2_{\widehat a}(\R)$.  On the other hand, 
   $\ker R_0(z) =0$     in $  L^2_{\widehat a}(\R)$ because $ \widetilde A_{[c,z ]} ^{(0)} R_0(z) $ is the identity operator  in  $  L^2_{\widehat a}(\R)$. Hence
\eqref {eq:notildeg} implies 
\begin{align*}
( \widetilde{\mathcal{L}} _{[c]} -     \widetilde {\mathbf{P}}_{[c]} )\widetilde g= z \widetilde g.
\end{align*}
Then $\widetilde{\mathcal{L}} _{[c]} \widetilde {\mathbf{P}}_{[c]} \widetilde g =(1+z)\widetilde {\mathbf{P}}_{[c]}  \widetilde g  $.
 which   by $\Re z \ge 0$  and the fact that $\sigma (\widetilde{\mathcal{L}} _{[c]} \widetilde {\mathbf{P}}_{[c]} )  = \sigma ( {\mathcal{L}} _{[c]}  {\mathbf{P}}_{[c]} ) =\{ 0\}$     implies $\widetilde {\mathbf{P}}_{[c]} \widetilde g=0$. 
In turn  
$  \mathcal{L}   _{[c]}  g=\lambda    g$ in the notation in \eqref{eq:deltac} and for $\lambda = \gamma  ^{3/2}  c \ z $. But this contradicts
{Proposition} \ref{prop:noembeig1} if $\lambda \neq 0$. Nor can we have $  \mathcal{L}   _{[c]}  g=0$    because this would imply 
$g= {\mathbf{P}}_{[c]}   g$, which would yield    $0\neq \widetilde g= \widetilde {\mathbf{P}}_{[c]} \widetilde g =0$.
Hence there is no such $\widetilde  g$ and we can  invert in \eqref{eq:smooth23large} obtaining 
\begin{align*}
   \left \|    (\widetilde{\mathcal{L}}_{[c]} - \widetilde {\mathbf{P}}_{[c]} -z)^{-1}    \right \| _{ \mathcal{L}  ( L ^{ 2,\sigma +1} \left ( \R\right ),   Y_s   )   } &\le    \left \|     (1   + K(z) ) ^{-1}  \right \| _{ \mathcal{L}  (    Y_s   )   }
   \left \|    R_0(z) (1-\gamma \partial ^2_x)   \right \| _{ \mathcal{L}  ( L ^{2,\sigma +1} \left ( \R\right ),   Y_s  )   } \\& \lesssim \left \|     (1   +  K(z)) ^{-1}  \right \| _{ \mathcal{L}  (   Y_s   )   },
\end{align*}
where we used \eqref{eq:R0-L1-W1inf-asymp}. By Lemma  \ref{lem:R0-L1-L2ms} the function  $z\to (1   + K(z) ) ^{-1}   $ is 
in $C^0(S ^+_\gamma , \mathcal{L}  (   Y_s   )) $ with
$(1   +  K(z) ) ^{-1}  \xrightarrow{z\to  \infty} 1 $  by \eqref{eq:R0-L1-tildesigma}. Then  \eqref{eq:smooth211} ensues.

  \end{proof}

\section{Proof of  Theorem \ref{prop:noemb_all}}\label{sec:noemb}

 In this section we prove the absence of  eigenvalues  in $\im \R \setminus \{ 0\}$.  Specifically, 
 for the proof of Theorem  \ref{prop:noemb_all}, it suffices to show that for any $p>1$, $c>1$ and $\tau\in \R\setminus\{0\}$, if $u\in D(\mathcal L_{[c]})= H^1(\R )$ satisfies 
\begin{align}\label{eq:eigen_eq}
    \mathcal L_{[c]}u=\im \tau u,
\end{align}
then $u=0$.
In the following, we fix $p>1$, $c>1$, $\tau\in \R\setminus\{0\}$ and $u\in H^1(\R )$ satisfying \eqref{eq:eigen_eq}.

\begin{lemma}
    For arbitrary $m\ge 1$, we have $u\in H^m(\R )$.
\end{lemma}

\begin{proof}
    We can rewrite \eqref{eq:eigen_eq} as follows and proceed by induction
    \begin{align}
        c\partial_x u =\partial_x(-\partial_x^2+1)^{-1}(u+pQ_c^{p-1}u)+\im \tau u.
    \end{align} 
\end{proof}

We set 
\begin{align}\label{eq:Martel_trans}
    v:=(-\partial_x^2+1)^{-1}L_c u.
\end{align}
Then, $v$ satisfies
\begin{align}
    L_c\partial_x v&= \im \tau (-\partial_x^2+1)v,\label{eq:v}\\
    \partial_x v&=\im \tau u.\label{eq:recov_u_from_v}
\end{align}
 If we can show $v=0$, \eqref{eq:recov_u_from_v} gives $u=0$.
Thus, it suffices to prove $v=0$. The key will be the  transformation of \eqref{eq:v} into a third order Fuchsian equation with three regular singular points.

\begin{lemma}
    Set $z=\mathrm{tanh}(\kappa x)$ with $\kappa=\frac{p-1}{2}\sqrt{\frac{c-1}{c}}>0$.
    Then, setting $w(z):=v(x)$, $w$ solves
    \begin{align}\label{eq:w}
        c\kappa^3((z^2-1)\partial_z)^3 w + \im \tau \kappa^2 ((z^2-1)\partial_z)^2 w -\kappa (z^2-1)(c-1+\gamma(z^2-1))\partial_z w -\im \tau w=0,
    \end{align}
    where  (only in this section)  $\gamma=(c-1)\frac{p(p+1)}{2}>0$.
\end{lemma}

\begin{proof}
    The conclusion \eqref{eq:w} follows from substituting  $\partial_x=\kappa(1-z^2)\partial_z$, $pQ_c^{p-1}(x)=\gamma(1-z^2)$ and $w(z)=v(x)$ into \eqref{eq:v}.
\end{proof}

We use the following consequence of the classical Frobenius theory
for regular singular points; see, e.g.,
\cite[Chapter IV, Sections 11--12]{HartmanBook}.
\begin{theorem}\label{thm:Fro}
    Let $N\ge 1$ and $\{a_j\}_{j=0}^N$ be holomorphic functions and $a_N\neq 0$ in a neighborhood of $0$.
    Assume $P(\rho):=\sum_{j=0}^N a_j(0)\rho^j$ has $N$ distinct roots $\rho_1,\cdots,\rho_N$ with $\Re \rho_1\ge \Re \rho_2\ge \cdots\ge \Re \rho_N$.
    Then, we can take $N$ independent solutions $\{u_j\}_{j=1}^N$
 of the ODE
 \begin{align}\label{eq:ODE_Fro}
     \sum_{j=0}^N a_j(z)\left(z\frac{d}{dz}\right)^j u=0,
 \end{align}
 of the form
 \begin{align}
     u_j (z)= z^{\rho_j}\sum_{k=1}^N (\log  z)^{k-1}   h_{j,k} (z)  ,
 \end{align}
 where $h_{j,k}$ are holomorphic in a neighborhood of $0$ with $h_{j,1}(0)=1$, $h_{j,k}(0)=0$ for $k\ge 2$ and $h_{1,k}\equiv 0$ for $k\ge 2$.
 \end{theorem}

 The equation for $w$ has the form \eqref{eq:ODE_Fro} at $z=\pm 1 $ and $\infty$ and is regular elsewhere.
 We investigate the possible asymptotics of $w$ around these singularities.

 We set $z_\pm=z\mp 1$ and $z_\infty=z^{-1}$ and $w_\sharp(z_\sharp)=w(z)$ for $\sharp =\pm,\infty$.

 \begin{lemma}\label{lem:Fr_w}
     For $\sharp=\pm,\infty$, $w_\sharp$ satisfies
     \begin{align}
         \sum_{j=0}^3 a_{\sharp,j}(z_\sharp)\left(z_\sharp\frac{d}{dz_\sharp}\right)^jw_\sharp(z_\sharp)=0,
     \end{align}
     where 
     \begin{align*}
a_{\pm,3}(z)&=c\kappa^3(z\pm2)^3,
\\
a_{\pm,2}(z)&=\kappa^2(z \pm2)^2
\bigl(3c\kappa z +\im \tau\bigr),
\\
a_{\pm,1}(z)&=\kappa(z\pm2)
\Bigl[
(2c\kappa^2-\gamma)z^2
+\bigl(\pm2(c\kappa^2-\gamma)+\im \tau\kappa\bigr)z
-(c-1)
\Bigr],
\\
a_{\pm,0}(z)&=-\im \tau,
\end{align*}
and
\begin{align*}
a_{\infty,3}(s)&=c\kappa^3(1-s^2)^3,\\
\quad a_{\infty,2}(s)
&=-\kappa^2(1-s^2)^2
\bigl[3c\kappa(1+s^2)+\im \tau s\bigr],\\
a_{\infty,1}(s)
&=\kappa(1-s^2)\Bigl[
2c\kappa^2(1+s^4)
+\im 
\tau\kappa s(1+s^2)-(c-1)s^2-\gamma(1-s^2)\Bigr],\\
a_{\infty,0}(s)&=\im \tau s^3.
\end{align*}

 \end{lemma}

\begin{proof}
    The claim follows by direct computation.
\end{proof}

From Lemma \ref{lem:Fr_w}, we can apply Theorem \ref{thm:Fro} to $w_\sharp$.
We set 
\begin{align}
    P_\sharp(\rho):=\sum_{j=0}^3 a_{\sharp, j}(0)\rho^j.
\end{align}
The properties of the roots of $P_\sharp$ given in the following lemma are the key.

\begin{lemma}\label{lem:root_key}
    For $\sharp=\pm,\infty$, $P_\sharp$ has three distinct solutions $\rho_{\sharp,k}$, $k=1,2,3$ satisfying 
    \begin{align}\label{eq:rootsorder}
        \Re \rho_{\sharp,1}>0=\Re \rho_{\sharp,2}>\Re \rho_{\sharp,3}.
    \end{align}
    Moreover, we have
    \begin{align}
        \Re \rho_{+,1}+\Re \rho_{-,1}+\Re \rho_{\infty,3}>0.\label{eq:key_rel}
    \end{align}
\end{lemma}

\begin{proof}
Since $f(s):= cs-\frac{s}{1+s^2}$ is strictly increasing, $f(0)=0$ and $\lim_{s\to \pm\infty}f(s)=\pm\infty$,
there exists a unique $\xi\in\mathbb R\setminus\{0\}$ such that $f(\xi)=\tau$.
Set
\[
\beta=-\frac{\xi}{2c(1+\xi^2)},\qquad
\alpha=\left(1-\frac{1}{c(1+\xi^2)}-\beta^2\right)^{1/2}.
\]
Then
\begin{align}\label{eq:alpha_est}
\alpha^2-\frac{c-1}{c}
=\frac{\xi^2\bigl(4c(1+\xi^2)-1\bigr)}
       {4c^2(1+\xi^2)^2}>0.    
\end{align}
Direct substitution gives
\[
(\rho_{\pm,1},\rho_{\pm,2},\rho_{\pm,3})
=\frac1{2\kappa}
 \bigl(\alpha\mp \im \beta,\ \mp \im \xi,\ -\alpha\mp \im \beta\bigr),
\]
and
\[
(\rho_{\infty,1},\rho_{\infty,2},\rho_{\infty,3})
=\left(3+\frac2{p-1},\,0,\,-\frac2{p-1}\right).
\]
These roots satisfy \eqref{eq:rootsorder}.
Moreover, since
$2\kappa=(p-1)\sqrt{(c-1)/c}$, by \eqref{eq:alpha_est} we obtain
\begin{align*}  
\Re\rho_{+,1}+\Re\rho_{-,1}+\Re\rho_{\infty,3}
&=\frac{\alpha}{ 2\kappa}  + \frac{\alpha}{ 2\kappa} -\frac2{p-1}  
=\frac{\alpha-\sqrt{(c-1)/c}}{\kappa}>0.    
\end{align*}
This completes the proof.
\end{proof}
Recall that $v\in L^2$ and $w(z)=v(x)$ with $z=\tanh(\kappa x)$.
This implies
\begin{align}\label{eq:w_integrability}
    \int_{-1}^1|w(z)|^2\frac{dz}{1-z^2}=\kappa \int_{\R}|v(x)|^2\,dx<\infty.
\end{align}
  Theorem \ref{thm:Fro} and Lemmas \ref{lem:Fr_w} and \ref{lem:root_key} imply that near $z=\pm 1$, $w$ is a linear combination of  
  \begin{align*}
      w_{\pm,j}(z):=(z\mp 1)^{\rho_{\pm,j}}\sum_{k=1}^3 h_{\pm,j,k}(z\mp 1)(\log (z\mp 1))^{k-1}.
  \end{align*}
However, since $\Re \rho_{\pm,j}\le 0$ for $j=2,3$, we see that for $j=2,3$, $w_{\pm,j}$ do not satisfy the integrability condition \eqref{eq:w_integrability}.
Therefore  $$w(z)=c_\pm w_{\pm,1}(z)=c_\pm(z\mp 1)^{\rho_{\pm,1}}h_{\pm,1,1}(z\mp1),$$
 near $z=\pm 1$ respectively for some $c_\pm\in \C$.

\begin{lemma}\label{lem:f_entire}
    Set $f(z):=(z- 1)^{-\rho_{+,1}}(z+1)^{-\rho_{-,1}}w(z)$ initially defined on the neighborhood of the origin.
    Then, $f$ can be analytically extended as an entire function.
\end{lemma}

\begin{proof}
Since Equation \eqref{eq:w} has singularities only at  $\pm1$, its solutions admit analytic continuation along
every path in $\mathbb C\setminus\{-1,1\}$, see \cite[Chapter 16, pp.~323--327]{rudin}. The same holds for $f$. Since   $f=c_\pm(z\pm 1)^{-\rho_{\mp,1}}h_{\pm,1,1}(z\mp1)$ near $\pm 1$,  for $h_{\pm,1,1}(z)$  holomorphic for $z$  near 0,  $f$ extends holomorphically
in a neighborhood of 
$\pm 1$. Then $f$ extends holomorphically to the whole $\C$. 
\end{proof}

\begin{proof}[Proof of Theorem \ref{prop:noemb_all}]
By Theorem \ref{thm:Fro} and Lemmas \ref{lem:Fr_w} and \ref{lem:root_key} for $\sharp=\infty$, we have
\begin{align}
    |w(z)|\lesssim |z|^{-\Re \rho_{\infty,3}}(\log |z|)^2,
\end{align}
for $|z|\gg 1$.
Thus, for the limit using \eqref{eq:key_rel},     we have
\begin{align}
    |f(z)|\lesssim |z|^{-\Re(\rho_{+,1}+\rho_{-,1}+\rho_{\infty,3})}(\log|z|)^2\xrightarrow{z\to  \infty} 0   .
\end{align} 
Since $f$ is an entire function, by Liouville's theorem   $f\equiv 0$.
 Then    $v\equiv 0$ and we have the conclusion.
\end{proof}

\section*{Acknowledgments}
M.  was supported by the JSPS KAKENHI Grant Number 23H01079 and 24K06792.

\section*{Statement of the use of AI}
The authors used Chat GPT  exclusively for editorial assistance and proofreading.
  All suggestions were reviewed
by the authors, who take full responsibility for the final manuscript.

	\begin{footnotesize}
		\bibliography{references}
		\bibliographystyle{plain}
	\end{footnotesize}

Department of Mathematics, Informatics and Geosciences,  University
of Trieste, via Valerio  12/1  Trieste, 34127  Italy.
{\it E-mail Address}: {\tt scuccagna@units.it, delsanto@units.it }

Department of Mathematics and Informatics,
Graduate School of Science,
Chiba University,
Chiba 263-8522, Japan.
{\it E-mail Address}: {\tt maeda@math.s.chiba-u.ac.jp}

\end{document}